\documentclass[12pt]{amsart}
\usepackage{fancyhdr}
\usepackage{graphicx}
\usepackage{amsmath}
\usepackage{mathrsfs}
\usepackage[all]{xy}
\usepackage{amssymb}
\usepackage{mathrsfs}
\usepackage{amscd}
\usepackage{color}
\usepackage{amsthm}
\usepackage{tikz-cd}
\usepackage{enumerate}
\usepackage{url}

\makeatletter
\def\@tocline#1#2#3#4#5#6#7{\relax
  \ifnum #1>\c@tocdepth 
  \else
    \par \addpenalty\@secpenalty\addvspace{#2}%
    \begingroup \hyphenpenalty\@M
    \@ifempty{#4}{%
      \@tempdima\csname r@tocindent\number#1\endcsname\relax
    }{%
      \@tempdima#4\relax
    }%
    \parindent\z@ \leftskip#3\relax \advance\leftskip\@tempdima\relax
    \rightskip\@pnumwidth plus4em \parfillskip-\@pnumwidth
    #5\leavevmode\hskip-\@tempdima
      \ifcase #1
       \or\or \hskip 1em \or \hskip 2em \else \hskip 3em \fi%
      #6\nobreak\relax
    \hfill\hbox to\@pnumwidth{\@tocpagenum{#7}}\par
    \nobreak
    \endgroup
  \fi}
\makeatother

{\theoremstyle{definition}
\newtheorem{Def}{Definition}[section]
\newtheorem{Rem}[Def]{Remark}
\newtheorem{Ex}[Def]{Example}
\newtheorem{sss}[Def]{}
}
\newtheorem{Prop}[Def]{Proposition}
\newtheorem{Thm}[Def]{Theorem}
\newtheorem{Lem}[Def]{Lemma}
\newtheorem{Cor}[Def]{Corollary}

\newtheorem{Conj}[Def]{Conjecture}

\newcommand{\R}{\mathbb{R}}
\newcommand{\C}{\mathbb{C}}

\newcommand{\Z}{\mathbb{Z}}
\newcommand{\Q}{\mathbb{Q}}
\newcommand{\N}{\mathbb{N}}

\newcommand{\spec}{{\rm Spec}}
\newcommand{\Star}{{\rm Star}}

\newcommand{\wt}{\widetilde}
\newcommand{\wtf}{{\rm wt}_\omega}
\newcommand{\ol}{\overline}
\newcommand{\A}{\mathscr{A}}
\newcommand{\red}{\mathrm{red}}
\newcommand{\M}{\mathscr{M}}
\newcommand{\sL}{\mathscr{L}}
\newcommand{\cX}{\ensuremath{\mathscr{X}}}

\newcommand{\an}{\mathrm{an}}
\newcommand{\tb}{\tilde{b}}
\newcommand{\sM}{\mathscr{M}}
\newcommand{\str}{\mathcal{O}}
\newcommand{\Sk}{\mathrm{Sk}}
\newcommand{\cC}{\mathcal{C}}
\newcommand{\fX}{\mathfrak{X}}
\newcommand{\sC}{\mathscr{C}}
\newcommand{\DEG}{{\rm \bf DEG}_{\rm ample}^{\rm split}}
\newcommand{\DD}{{\rm \bf DD}_{\rm ample}^{\rm split}}

\newcommand{\For}{{\rm \bf For}}
\newcommand{\f}{{\rm \bf for}}
\newcommand{\Pm}{\mathscr{P}}
\newcommand{\Pt}{{\tilde{\mathscr{P}}}}
\newcommand{\rig}{{\rm \bf rig}}
\newcommand{\ber}{{\rm \bf ber}}
\newcommand{\spf}{{\rm Spf}}
\newcommand{\cXt}{{\tilde{\cX}}}
\newcommand{\rsp}{{\rm Sp}}

\begin{document}
\pagestyle{fancy}
\lhead[\thepage{} Keita Goto]{}
\rhead[]{On the two types of affine structures for degenerating Kummer surfaces \thepage{}}
\cfoot[]{}
\normalfont
\title{On the two types of affine structures for degenerating Kummer surfaces \\ --Non-archimedean vs Gromov-Hausdorff limits--}
\author{Keita Goto}
\address{Department of Mathematics, Kyoto university, Oiwake-cho, Kitashirakawa, Sakyo-ku, Kyoto city,
Kyoto 606-8285. JAPAN}
\email{k.goto@math.kyoto-u.ac.jp}
\thanks{This work is supported by JSPS KAKENHI Grant Number JP20J23401.}
\date{\today}
\begin{abstract}

Kontsevich and Soibelman constructed integral affine manifolds with singularities (IAMS, for short)
  for maximal degenerations of polarized Calabi-Yau manifolds
  in a non-Archimedean way.
  On the other hand, for each maximally degenerating family of polarized Calabi-Yau manifolds,
  we can consider the Gromov-Hausdorff limit of the fibers.
It is expected that
this Gromov-Hausdorff limit carries an IAMS-structure.
   Kontsevich and Soibelman conjectured that these two types of IAMS
   are the same. This conjecture is believed in the mirror symmetry context.
   In this paper, we prove the above conjecture for maximal degenerations of
   polarized Kummer surfaces.
\end{abstract}
\maketitle
\thispagestyle{empty}
\tableofcontents

\section{Introduction}
\begin{sss}
At the end of the $20$th century, in order to formulate what is called mirror symmetry,
several approaches have been proposed.
One of them is due to Strominger, Yau and Zaslow \cite{Strominger1996MirrorSI}.
In \emph{op.cit.},
they gave a geometric interpretation for mirror symmetry
and proposed a conjecture called the SYZ conjecture.
Gross and Siebert provided an algebro-geometric interpretation
 of the SYZ conjecture \cite{article}. It is known as the Gross-Siebert program.
In this program, it is important to construct an integral affine manifold with singularities (IAMS,for short) from degeneration
of polarized Calabi-Yau manifolds, and vice versa.
For a (toric) degeneration
of polarized Calabi-Yau manifolds,
they extracted the polyhedral decomposition and the fan structure for each vertex
and gave an IAMS structure to the dual intersection complex based on them, and vice versa.

Kontsevich and Soibelman constructed an IAMS
structure of the dual intersection complex in a non-Archimedean way \cite{Kontsevich2006}.
The exact definition will be given later (\S 4), but for now, we call it \emph{non-Archimedean SYZ Picture}.
In [\emph{op.cit.}, \S4.2],
they mentioned
the specific IAMS structure
for
 the degeneration of
 K3 surfaces defined by
$$\{x_0x_1x_2x_3+tP_4(x)=0 \}\subset \mathbb{P}^3\times \Delta,$$
where
$x=[x_0:x_1:x_2:x_3]$ are homogeneous coordinates on
$\mathbb{P}^3$,
$\Delta$ is a (formal) disk with a (formal) parameter $t$
and $P_4$ is a generic homogeneous polynomial of degree $4$.
For general degenerations of  K3 surfaces,
however, the specific affine structures constructed in this way are not well known.

The main goal of this paper is to reveal the IAMS structure
constructed in the non-Archimedean SYZ Picture
for degenerations of Kummer surfaces.
Further, we clarify the sense of what is called `Collapse Picture' in
\cite{Kontsevich2006}
 related to the Gromov-Hausdorff limit (In this paper,
we also call it
 \emph{`Gromov-Hausdorff limit Picture'}) and
prove the following conjecture appeared in \emph{op.cit.} for degenerations of Kummer surfaces.
\begin{Conj}[{\cite[Conjecture 3]{Kontsevich2006}}]
For maximal degenerating polarized algebraic Calabi-Yau varieties,
the IAMS structure induced by
Collapse Picture
coincides with the IAMS structure induced by non-Archimedean SYZ Picture.
\end{Conj}
That is, the following is the main theorem of this paper.
\begin{Thm}[{= Theorem \ref{main}, \cite[Conjecture 3]{Kontsevich2006}} for Kummer surfaces]\label{main thm}
  For maximal degenerations of polarized Kummer surfaces,
  the IAMS structure induced by
  Gromov-Hausdorff limit Picture
  coincides with the IAMS structure induced by non-Archimedean SYZ Picture up to scaling.
\end{Thm}
In the process of proving this, we prove
that the non-Archimedean SYZ Picture for polarized Kummer surfaces
 is explicitly described by  the degeneration data as in \cite{FC} (= Theorem \ref{nAside}).
We note that it does not depend on the polarization as we will state in Remark \ref{indep of pol}.
On the other hand, when we consider the Gromov-Hausdorff limit of the fibers, we need the polarization.
At first glance, the Gromov-Hausdorff limit Picture seems to depend on the polarization,
however, the above Theorem \ref{main thm} implies that it does not depend on the polarization.
In addition, we prove that \cite[Conjecture 3]{Kontsevich2006} for abelian surfaces also holds (=Theorem \ref{Conj3 for as}).
\end{sss}
\begin{sss}
Here is a brief description of the structure of this paper.
In $\S 2$, we introduce some notation and collect some basic facts for subsequent discussions.
In $\S 3$, we recall K\"{u}nnemann's construction
of
 projective models of abelian varieties with finite group actions.
It is a modification of Mumford's construction by which we can construct a semiabelian degeneration
from degeneration data. In K\"{u}nnemann's construction, it is important to construct a cone decomposition
associated with the degeneration. Further, the cone decomposition is also
important for non-Archimedean SYZ Picture as we will see.
In applying
K\"{u}nnemann's construction
 to the proof of our main theorem, we will
modify his method due to technical problems
 (= Lemma \ref{key subdivision}).
In $\S 4$, we recall non-Archimedean SYZ fibration that is
originally introduced in \cite{Kontsevich2006}. We use terminology based on \cite{nicaise_xu_yu_2019}.
In \emph{op.cit.},
the authors dealt with `good' minimal dlt-models with a technical assumption. However, since
we will only deal with snc-models here,
some definitions have been simplified accordingly.
In \emph{op.cit.},
they
proved that a singular locus of an IAMS induced by non-Archimedean SYZ fibration
is of codimension $\geq 2$. Further,
they proved the uniqueness of what they call  \emph{piecewise} integral affine
 structures.
 It is more of topological structures.
In contrast,
 our results of this paper focus on the IAMS structure,
 and we describe it explicitely.
 In $\S 5$, we prove Theorem \ref{main thm}.
 To this end, we combine the tools introduced in the previous
 sections to show what exactly happens to
 the non-Archimedean Picture of an Abelian variety with a finite group action.
 Further, we define the Gromov-Hausdorff limit Picture
 and observe some properties about it.
 By comparing the two IAMS structures for degenerations of
 Kummer surfaces, we prove Theorem \ref{main thm}.
\end{sss}

 \begin{center}
   { \large{Acknowledgements}}
 \end{center}
 I would like to thank
 Professor \emph{Yuji Odaka}
  for a lot of suggestive advice and productive discussions.
I learned
radiance obstructions and integral affine structures from
Dr.
\emph{Yuto Yamamoto}
and
Mr. \emph{Yuki Tsutsui}.
I am grateful to them.
 This work is supported by JSPS KAKENHI Grant Number JP20J23401.

\section{Preliminaries and Notation}
\begin{sss}
  In this paper, we fix the notation as follows:
  Let $R$ be a complete discrete valuation ring (cDVR, for short) with uniformizing parameter $t$ and
   algebraically closed residue  field $k$.
We note that we start with the residue field $k$ of an arbitrary characteristic,
 but later we make the condition stronger.
   Let $S=\spec R$, and let $\eta$ be the generic point of $S$.
We denote by $K=\str_{S,\eta}$ the fraction field of $R$.
  Let $|\cdot|$ be the valuation on $K$ uniquely determined by $|t|=e^{-1}$.
\end{sss}

\begin{Def}
Let $X$ be a locally Noetherian scheme and let $D$ be an effective Cartier divisor on $X$. Let $D_1,..., D_r$ be the irreducible components of $D$ endowed with the reduced induced closed subscheme structure. For each subset $J\subseteq \{1,...,r\},$ we denote by $D_J$ the scheme-theoretic intersection $\cap_{j\in J} D_j$.
If $J=\emptyset$, we note $D_{\emptyset}:=X.$\\
An effective divisor $D$ on $X$ is said to be with  \emph{strict normal crossings} if it satisfies the following.
\begin{itemize}
  \item
  $D$ is reduced.
  \item
  For each point $x$ of $D$, the stalk $\str_{X,x}$ is regular.
  \item
  For each nonempty set
  $J\subseteq \{1,..., r\}$, the scheme $D_J$ is regular and of codimension $|J|$ in $X$.

  \end{itemize}

\end{Def}

\begin{sss}
  Let $X$ be a smooth $K$-variety, where the characteristic of $K$ is assumed to be zero. A {\em model} of $X$ is a flat $R$-algebraic space $\cX$
  endowed with an isomorphism $\cX_K (=\cX\times_S \spec K)\to X$.
  (We do not assume properness and quasi-compactness.)
    An \emph{snc-model of $X$} (or, a \emph{semistable model of $X$}) is a regular model $\cX$ of $X$ such that $\cX$ is a scheme and the central fiber $\cX_k(=\cX\times_S \spec k)$ is a divisor with strict normal crossings.
    By the semistable reduction theorem
     \cite[Chapter 4 \S3]{kempf1973toroidal}, there exists a finite extension $K'$ of $K$ such that $X\times_K K'$ has an
      snc-model over the integral closure of $R$ in $K'$.
     Further, if $X$ is projective, then we can
     obtain a projective snc-model.
\end{sss}
\begin{Def}[Kulikov Model]
   Let $X$ be a geometrically integral smooth projective variety over $K$ with $\omega_X \cong \str_X$. A \emph{Kulikov model} of $X$ is a regular algebraic space $\mathscr{X}$ that is proper and flat over $S$ with the following properties:
   \begin{itemize}
     \item
      The algebraic space $\mathscr{X}$ is a model of $X$.
     \item
      The special fiber  $\mathscr{X}_k$ of $\mathscr{X}$ is a reduced scheme.
     \item
     The special fiber $\mathscr{X}_k$   has strict normal crossings on $\mathscr{X}.$
     \item
      $\omega_{\mathscr{X}/S}$ is trivial.\label{Kulikovdefinition}
   \end{itemize}
\end{Def}
\begin{sss}
  If a Kulikov model $\cX$ of $X$ is a scheme, then $\cX$ is an snc-model.
\end{sss}
\begin{Def}
  A \emph{stratification} of a scheme $X$ is a \emph{not necessarily} finite set $\{X_{\alpha}\}_{\alpha \in I}$ of locally closed subsets, called the \emph{strata}, such that
  every point of $X$ is in exactly one stratum, and such that the closure of a stratum is a finite union of strata.
   We note that a stratification  in the sense of \cite[(1.3)]{Kun} (or \cite[p.56]{kempf1973toroidal}) had to be a finite set.
\end{Def}
\begin{sss}
  For an snc-model $\cX$,  the special fiber $\cX_k$ induces a stratification of $\cX_k$ naturally.
We denote by $\Delta (\cX)$ the \emph{dual intersection complex} of the special fiber $\cX_k$ with respect to this stratification.
\end{sss}

\begin{sss}
  For an $R$-scheme $\cX$, we denote by $\cX_\f$  \emph{the formal completion of $\cX$ along the special fiber $\cX_k$}.
  If $\cX$ is covered by  open affine subschemes of the form  $\spec A_\alpha$, the formal completion
  $\cX_\f$ is obtained by glueing open formal subschemes of the form
  $\spf \hat{A}_\alpha$ together, where
  $\hat{A}_\alpha$ is the $t$-adic completion of $A_\alpha$.
  In particular, for flat $R$-scheme $\cX$ locally of finite type,
  the formal completion $\cX_\f$ is a flat formal $R$-scheme locally of finite type.
  Here, a \emph{flat formal $R$-scheme locally of finite type} (resp. admissible formal $R$-scheme)
  means that it is covered by \emph{not necessarily} finitely many (resp. finitely many) open formal subschemes
  of the form $\spf \A_\alpha$, where $\A_\alpha$ is  an admissible $R$-algebra.
\end{sss}
\begin{sss}\label{Ray}
  For $R$-algebra $\A$,
  we write $\A_K$ (resp. $\A_k$) instead of $\A \otimes_R K$ (resp. $\A \otimes_R k$).
   We can consider
   a functor called
   \emph{the Raynaud generic fiber}

$-_{\rig}:$ \{flat formal $R$-schemes locally of finite type\} $\to$ \{rigid $K$-spaces\}.

The functor is constructed by sending
an affine admissible formal $R$-scheme $\spf \A$
to
the $K$-affinoid space $\rsp \A_K$,
where the underlying space of $\rsp \A_K$ is the set ${\rm Max} \A_K$ of all maximal ideals of $\A_K$
equipped with the weak topology with respect to $\A_K$ and the $G$-topology (cf. \cite[9.1.4]{BGR}).
This functor first appeared in \cite{MSMF_1974__39-40__319_0}.
It is known that this functor preserves fiber products.
Let
$f: \fX \to \mathfrak{Y}$
be a morphism
 between formal $R$-schemes locally of finite type.
 If $f: \fX \to \mathfrak{Y}$ is a finite morphism
 (resp. closed immersion, open immersion, immersion, separated morphism), then
$f_{\rig}: \fX_{\rig} \to \mathfrak{Y}_{\rig}$ is a finite morphism
(resp. closed immersion, open immersion, immersion, separated morphism).
\end{sss}

\begin{sss}
Berkovich gave the fully faithful functor

$-_{0}:$ \{separated strictly $K$-analytic spaces \} $\to$ \{rigid $K$-spaces\}

\noindent in the process of basing his analytic spaces (cf.
\cite[\S 3.3]{Berk90}).
This functor preserves fiber products. In addition,
the following also holds.
\begin{Prop}[{\cite[Proposition 3.3.2]{Berk90}}]\label{RB}
  Let $f:X \to Y$ be a morphism between separated strictly $K$-analytic spaces,
$f: X \to Y$ is a finite morphism
(resp. closed immersion, open immersion, immersion, separated morphism)
if and only if
$f_{0}: X_{0} \to Y_{0}$ is a finite morphism
(resp. closed immersion, open immersion, immersion, separated morphism).
\end{Prop}
For a flat  formal $R$-scheme $\fX$ locally of finite type,
there is a unique strictly $K$-analytic space $X$ such that
$\fX_\rig \cong X_0.$
For simplicity of notation, we use the letter $\fX_\ber$ for this $X$.
In particular, a $K$-affinoid space $\rsp \A_K$ corresponds to the Berkovich Spectrum $\M (\A_K)$,
where $\M (\A_K)$ is the set of all bounded multiplicative seminorm on $\A_K$
equipped with the weak topology with respect to $\A_K$ and  the $G$-topology (cf. \cite[\S2, \S3]{Berk90}).
We note that $\M (\A_K)\subset X$ is closed but not necessarily open,
although $\rsp \A_K\subset X_0$ is a closed and open set.
That is, we regard the Raynaud generic fiber as the functor from the category of  flat formal $R$-schemes
locally of finite type
to the category of separated strictly $K$-analytic spaces.
By abuse of notation, we write $\cX_\ber$ for $(\cX_\f)_\ber$
for  a flat $R$-scheme $\cX$ locally of finite type.
\end{sss}
\begin{Def}\label{def of reduction}
  Let
  $\fX$
  be
 a flat formal $R$-scheme
  locally of finite type.
    Then we can consider \emph{the reduction map $\red_\fX : \fX_\ber \to \fX$}.
    Locally this map
    $\red_\fX|_{\M(\A_K)} : \M(\A_K) \to \spf \A=\spec \A_k$
    is defined as follows:
A point $x\in \M(\A_K)$ can be seen as a
multiplicative seminorm on $\A_K$ that is bounded by the equipped norm on $\A_K$. Since $\A$ is an admissible $R$-algebra, the restriction of
the equipped norm on $\A_K$ to $\A$ is bounded by 1.
Hence, the restriction of $x$ to $\A$ is also bounded by 1.
Then,
$$\mathfrak{p}_x:=\{f\in \A\ |\ |f(x)|<1\} \subset \A$$
is a prime ideal of $\A$. It is clear that $\mathfrak{p}_x \in \spec \A_k =\spf \A$.
Then we denote by $\red_\fX(x)$ the point corresponding to this prime ideal $\mathfrak{p}_x$.
If $\fX=\cX_\f$ for some flat $R$-scheme $\cX$ locally of finite type,
we write $\red_\cX$ instead of $\red_{\fX}$.
    In the author's previous work \cite{goto2020berkovich},
    the image $\red_\cX(x)$ by the reduction map is called
    the center of $x$.

\end{Def}
\begin{sss}
  Let
  $\fX$
  be
 a flat formal $R$-scheme
  locally of finite type.
    Then the reduction map
$\red_\fX : \fX_\ber \to \fX$ is
 anti-continuous
 and
 surjective.

Please refer to \cite[\S 2.4]{Berk90} for details.

\end{sss}
\begin{Def}
\label{affine}
Let $B$ be an real $n$-dimensional manifold.
An \emph{affine structure} (resp.   \emph{integral affine structure}) on $B$ is an atlas $\{(U_i,\psi_i)\}$ of $B$
consisting of coordinate charts $\psi_i:U_i\rightarrow \R^n$,
whose transition functions $\psi_i\circ\psi_j^{-1}$ lie in
${\rm Aff}(\R^n):=\R^n\rtimes {\rm GL}(\R^n)$ (resp. ${\rm Aff}(\Z^n):=\Z^n\rtimes {\rm GL}(\Z^n)$).
A pair of $B$ and an affine structure
 (resp. integral affine structure)
 on $B$
 is called an \emph{affine manifold} (resp. an  \emph{integral affine manifold}).
Further, $B$ is called an \emph{integral affine
manifold with singularities} (\emph{IAMS}, for short) if
$B$ is a $C^0$-manifold with an open set $B^{\rm sm} \subset B$ that has an
integral affine structure, and such that $Z:=B\setminus B^{\rm sm}$ is a locally
finite union of locally closed submanifolds of codimension $\geq 2$.
Here, $Z$ is called a \emph{singular locus} of $B$.
The pair of this integral affine manifold $B^{\rm sm}$ and $B$ is called \emph{IAMS structure} of $B$.
\end{Def}

\section{Degenerations of Kummer surfaces}
First, we introduce some important results from \cite{Kun} (cf. \cite{FC}).
\begin{sss}
Let $G$ be a \emph{semiabelian scheme} over $R$. That is, \emph{$G$ is a smooth separated
group scheme of finite type over $S$ whose geometric fibers are extensions of
Abelian varieties by algebraic tori}. We assume
that $G_\eta$ is Abelian variety. Let $\sL$ be a line bundle on $G$
such that $\sL_\eta$ is ample on $G_\eta$. Then we obtain \emph{the Raynaud extension}
\[0\to T\to \tilde{G} \xrightarrow{\pi} A \to 0\]
associated with $G$ and $\sL$, where $T$ is an algebraic torus, $A$ an Abelian
scheme, and $\tilde{G}$ a semiabelian scheme over $S$.
If the abelian part $A$ is trivial, $G$ is called \emph{maximally degenerated}.
We note that we need to choose such a line bundle $\sL$ to obtain this Raynaud extension.
However, this extension is independent of the choice of $\sL$.
The line bundle $\sL$ induces a line bundle $\tilde{\sL}$ on $\tilde{G}$.
We assume that all line bundles have cubical structures as well as
\cite[(1.7)]{Kun}.
In this paper, we shall use the categories $\DEG$ and $\DD$ introduced by
\cite{Kun}.
Each category is a subcategory
of ${\rm \bf DEG}_{\rm ample}$ and ${\rm \bf DD}_{\rm ample}$ as constructed in \cite{FC}, respectively.
In particular, there is an equivalence of categories
${\rm \bf M}_{\rm ample}: {\rm \bf DD}_{\rm ample} \to {\rm \bf DEG}_{\rm ample}$
(See \cite[Chapter III, Corollary 7.2]{FC}).
We denote $F_{\rm ample}$ by the inverse of this functor.
Originally, $F_{\rm ample}$ is a more naturally determined functor,
and its inverse, ${\rm \bf M}_{\rm ample}$, is the non-trivial functor.

Objects of the category $\DEG$ of split ample degenerations are triples $(G,\sL, \sM)$,
where $G$ is a semiabelian scheme over $S$ such that $T$ is a \emph{split torus} over $S$,
$\sL$ a cubical invertible sheaf on $G$ such that $\sL_\eta$ is ample on $G_\eta$,
 and $\sM$ a cubical ample invertible sheaf on $A$ such that
 $\tilde{\sL} = \pi^* \sM$.
In particular, $\sM$ is trivial when $G$ is maximally degenerated.
By definition of the algebraic torus,
 every ample degeneration $(G,\sL)$ becomes split after a finite extension of
 the base scheme $S$.

 On the other hand, objects of the category $\DD$ of split ample degeneration data
 are tuples
 \[(A, M, L, \phi, c, c^t, \tilde{G} , \iota, \tau, \tilde{\sL}, \sM, \lambda_A,\psi, a, b).
 \]
 Here, $M$ and $L$ are free Abelian groups of the same finite rank $r$, and $\phi : L\to M$
 is an injective homomorphism.
  Functions $a:L\to \Z$ and $b:L\times M \to \Z$  are
 determined by $\psi$ and $\tau$, respectively.
 We note that $M$ reflects the information
 of the Raynaud extension (or more precisely, its split torus part),
 $\phi$ reflects the information of polarization,
 $a$ and $b$ reflect the information of $G_\eta$-action.
 In particular,
 $(G,\mathscr{L})$ is called \emph{principally polarized} if
 the morphism $\phi$ induced by ${\rm \bf M}_{\rm ample}$ is an isomorphism.
 Since we will not use the rest in this paper,
 the rest is omitted.
 Please refer to \cite{Kun} for more details.

  We note that there is an equivalence of categories
 $F:\DEG \to \DD$ (cf. \cite[(2.8)]{Kun}). This functor is defined by
 the restriction of $F_{\rm ample}={\rm \bf M}_{\rm ample}^{-1} : {\rm \bf DEG}_{\rm ample} \to  {\rm \bf DD}_{\rm ample} $ to $\DEG$.

\end{sss}
 \begin{sss}
 The key idea of \cite{Kun} is to construct rational polyhedral cone decompositions
 that give us the relatively complete model as in \cite{CM_1972__24_3_239_0}.
 To construct them, we shall use the category $\cC$ introduced
 by \cite[\S3]{Kun} (cf. \cite{overkamp2021}).

Objects of the category $\cC$ are tuples $(M, L, \phi, a, b)$, where $M$ and $L$ are free
Abelian groups of the same finite rank, $\phi :L\to M$ is an injective homomorphism, $a: L\to \Z$
is a function with $a(0)=0$, and $b:L\times M \to \Z$ is a bilinear pairing such that $b(-,\phi(-))$
is symmetric, positive definite, and satisfies
\[
a(l+l')-a(l)-a(l')=b(l,\phi (l')).
\]

There is a natural forgetful functor
$\For : \DD \to \cC$.
This function extracts the information necessary to construct
rational polyhedral cone decompositions from the degeneration data $\DD$.
\end{sss}
\begin{sss}\label{base change}
    We set $S'=\spec R'$, where
  $R'$ is another cDVR and $\eta'$ is its generic point.
  Let $f:S'\to S$ be a finite flat morphism, let $\nu$ be the degree of
  $f^*: K=\str_{S,\eta}\hookrightarrow K'=\str_{S',\eta'}$.

  In fact, $\DEG$ and $\DD$ depend on the base field $K$.
  That is, $\DEG$ (resp. $\DD$) should have been written as ${\rm \bf DEG}_{{\rm ample},K}^{\rm split}$
  (resp. ${\rm \bf DD}_{{\rm ample},K}^{\rm split}$).
  In particular, these categories are not closed under base change along $f:S'\to S$.
  However, since we are dealing with degenerations after
  sufficient finite extension, these abbreviations do not cause any problem.

  On the other hand, $\cC$ does not depend on the base field $K$.
  Let us see what happens when we take the base change along $f:S'\to S$.

  Given $(G,\mathscr{L},\M)\in {\rm \bf DEG}_{{\rm ample},K}^{\rm split}$, let
  $(G',\mathscr{L}',\M')\in {\rm \bf DEG}_{{\rm ample},K'}^{\rm split}$
  be the base change
  of $(G,\mathscr{L},\M)$
  along $f:S'\to S$.
  If $\For(F(G,\mathscr{L},\M))=(M, L, \phi, a, b)\in\cC,$
  then $\For(F(G',\mathscr{L}',\M')) \cong (M, L, \phi, \nu \cdot a,\nu\cdot  b)$ (cf. \cite[(2.9)]{Kun}).
\end{sss}
\begin{sss}\label{act}
Let $H$ be a finite group acting on $(G,\mathscr{L},\M)\in \DEG$.
It means that we can regard each $h\in H$ as the
$S$-automorphism $$h: (G,\mathscr{L},\M)\to (G,\mathscr{L},\M)$$
and these morphisms are compatible in a natural way.
Further, we also define the action of $H$ on $F((G,\mathscr{L},\M)) \in \DD$
 (resp. $\For(F((G,\mathscr{L},\M)))\in\cC$).
  See \cite[(2.10)]{Kun} for details.
 Note that we assume that $H$ acts trivially on $S$ in this paper.
\end{sss}
\begin{sss}\label{how to act}
  Given an object $\For(F(G,\mathscr{L},\M))=(M,L, \phi, a, b)\in \mathcal{C}$ on which the finite
  group $H$ acts as \ref{act}, we obtain an action (from the left) of $H$ on $L$, and
  an action (from the right) of $H$ on $M$. We set $\Gamma:=L \rtimes H$ and
 $\tilde{M}:=M\oplus \Z$.
 Then we denote by $N$ (resp. $\tilde{N}$) the dual of $M$ (resp. $\tilde{M}$).
 Let $\langle -,-\rangle : \tilde{M}\times \tilde{N} \to \Z $ be \emph{the canonical pairing}.

  Now we define the action of $\Gamma$ on $\tilde{N}=N\oplus \Z$ via
  $$S_{(l,h)}((n,s)):=(n\circ h+sb(l,-), s),$$ as in \cite[p.181]{Kun}.
  As we will now explain, this action reflects the natural
  action of $\Gamma$ on $T_\eta =\spec K[M]$, where $T$ is a split torus part of $\tilde{G}$.
 At first,  we identify $\tilde{m}=(m,k)\in \tilde{M}$ with $t^k X^m\in K[M]$.
  In the proof of \cite[Lemma 3.7]{Kun},
  the action of $L$ on $T_\eta = \spec K[M]$
  induced by the natural action of $T_\eta$
  is defined as follows:
  $$l : \tilde{M} \to \tilde{M},  \ \  \ (m,s) \mapsto (m,b(l,m)+s) . $$
  We can easily verify that
  this action is dual to the action $S_{(l,\mathrm{Id})}$
  in the sense of $\langle l\cdot \tilde{m},\tilde{n}	\rangle
  =\langle \tilde{m},S_{(l,\mathrm{Id})}(\tilde{n})\rangle$.
  In the same way, we can easily check that the action of $h\in H=\{\pm 1\}$ on $T$ is dual to $S_{(0,h)}$.
  Hence, the action of $\gamma \in \Gamma$ on $T$ corresponds to $S_\gamma$ on $\tilde{N}$.

  In addition, we consider the function $\chi: \Gamma \times \tilde{N}_\R\to \R$ defined by
  $$\chi ((l,h), (n,s))=sa(l)+n\circ\phi\circ h^{-1}(l)$$
as in \cite[p.181]{Kun}.

  In $\tilde{N}_\R = N_{\R} \oplus \R,$ we have the cone
  $\mathscr{C}:=(N_{\R}\oplus\R_{>0})\cup\{0\}.$
  The cone $\sC$ is stable under the action of $\Gamma$.
  We shall consider a smooth $\Gamma$-admissible rational polyhedral cone
  decomposition $\Sigma :=\{\sigma_\alpha\}_{\alpha\in I}$
  which admits a $\Gamma$-admissible $\kappa$-twisted polarization
  function $\varphi\colon \mathscr{C}=\bigcup_{\alpha\in I} \sigma_{\alpha}\to \R$
for some $\kappa \in \N$.
Let us take a moment to recall these definitions.
\begin{Def}
A rational polyhedral cone
decomposition $\Sigma:=\{\sigma_\alpha\}_{\alpha\in I}$
of $\sC$
is said to be
\emph{$\Gamma$-admissible} if 
the above action of $\Gamma$ on $\sC$ 
causes a bijection from $I$ to itself
(that is, for any $\sigma_\alpha \in \Sigma$ and any $\gamma \in \Gamma,$ there exists $\beta\in I$ such that $\sigma_\beta=\gamma (\sigma_\alpha)$ in $\sC$.)
  and
  we can take a system of finitely many representatives $\{\sigma_\alpha\}$
  for the action of $\Gamma$ (that is,
  there are at most finitely many orbits).
  
For any cone $\sigma_\alpha \in \Sigma$, there exists  $l_1,\dots l_d \in \tilde{N}$
such that $$\sigma_\alpha=\sum_{i=1}^d\R_{\geq 0}l_i.$$
In particular, 
each $l_i\in \tilde{N}$ corresponds to a $1$-dimensional face of $\sigma_\alpha$.
Here, note that each $l_i\in \tilde{N}$ is not necessarily primitive.
The cone $\sigma_\alpha=\sum_i\R_{\geq 0}l_i$ is said to be a \emph{simplex} if $l_1,\dots, l_d\in \tilde{N}$ are linearly independent in $\tilde{N}$.
Further, the cone $\sigma_\alpha=\sum_i\R_{\geq 0}l_i$ is said to be \emph{smooth} if 
$\sigma_\alpha$ is a simplex and
$l_1,\dots, l_d\in \tilde{N}$ form a part of a basis of $\tilde{N}$.
Similarly, the decomposition 
$\Sigma$ is said to be \emph{simplicial} if each $\sigma_\alpha\in \Sigma$ is simplex.
Further, the decomposition
$\Sigma$ is said to be \emph{smooth} if each simplex $\sigma_\alpha \in \Sigma$ is smooth.

A function $\varphi\colon \sC=\bigcup_{\alpha\in I} \sigma_{\alpha}\to \R$
is called \emph{polarization function}
associated with $\Sigma$ if it satisfies the following properties:
\begin{itemize}
  \item $\varphi$ is continuous  function that satisfies $\varphi (\tilde{N}\cap  \sC) \subset \Z$
  \item $\varphi (rx)=r\varphi (x)$, for any $r\in \R_{\geq 0}$
  \item The restriction $\varphi |_{\sigma_\alpha}$ to each cone $\sigma_\alpha$ is a linear function
  \item
  $\varphi$ is strictly convex function for $\Sigma$.
  That is, for any $\sigma \in \Sigma$,
  there exists $r\in \N$ and $\tilde{m}\in \tilde{M}$
  such that $\langle \tilde{m}, \tilde{n} \rangle  \geq r\varphi (\tilde{n})$ for all $\tilde{n}\in \sC$ and
   $\sigma = \{\tilde{n} \in \sC \ |\  \langle \tilde{m}, \tilde{n} \rangle = r\varphi (\tilde{n}) \}$

\end{itemize}
A polarization function $\varphi\colon \sC \to \R$ is called \emph{$\kappa$-twisted $\Gamma$-admissible}
 for some $\kappa\in \N$
if it satisfies
$\varphi(x) - \varphi \circ S_{\gamma}(x) =\kappa \chi (\gamma, x)$
for all $\gamma\in \Gamma$, $x\in \sC$.
When $\kappa$ is not important, it is often referred to as \emph{$\Gamma$-admissible polarization} for short.
\end{Def}
We denote by $I^d \subset I$ the set of the indices corresponding to the $d$-dimensional cones of $\Sigma$.
We set $I^+:= \bigcup_{d>0} I^d$. Since $\Sigma$ is $\Gamma$-admissible, the group $\Gamma$ acts on each $I^d$.
  Overkamp combines various Theorems and Propositions in \cite{Kun} into  the following result  \cite[Theorem 2.2]{overkamp2021}:
\end{sss}
\begin{Thm}[{\cite{Kun}}, {\cite[Theorem 2.2]{overkamp2021}}]
  We set a semiabelian degeneration
   $(G, \mathscr{L}, \mathscr{M})\in \DEG$ and assume that $H$ acts on this object as \eqref{act}.
   We denote by $\A$ the N\'{e}ron model of the Abelian variety $A:=G_\eta$.
    Let $(M,L,\phi, a,b):=\For(F((G, \mathscr{L}, \mathscr{M})))$ and suppose we have a smooth $\Gamma$-admissible rational polyhedral cone decomposition $\Sigma :=\{\sigma_{\alpha}\}_{\alpha\in I}$ of
  $\sC \subset \tilde{N}_\R$.
  Furthermore we assume that this decomposition $\Sigma$ has the following properties:

  \begin{enumerate}[(a)]
    \item     There exists a $\kappa$-twisted  $\Gamma$-admissible polarization function $\varphi$ for the decomposition $\Sigma$.
    \item  The decomposition $\Sigma$ is \emph{semistable}. That is, the primitive element of any one-dimensional cone of the decomposition $\Sigma$ is of the form $(n,1)$ for some $n\in N.$
    \item   The cone $\sigma_T=\{0\}\times\mathbf{R}_{\geq0}$ is contained in the decomposition $\Sigma$.
    \item   For all $l\in L\backslash \{0\}$ and $\alpha\in I$, it holds that
    $$\sigma_{\alpha}\cap S_{(l,\mathrm{Id})}(\sigma_{\alpha})=\{0\}.$$
  \end{enumerate}

  Then there exists a projective snc model $\Pm$ of $A$ over $S$ associated to $\Sigma$
 and a line bundle $\mathscr{L}_\Pm$
 such that the following holds:
 \begin{enumerate}[(i)]
  \item
  The canonical morphism
  $\Pm^{\rm sm}\to \mathscr{A}$ is an isomorphism.
  \item
  The action of $H$ on $G=\mathscr{A}^0$ extends uniquely to $\Pm$, and the restriction of $\mathscr{L}_\Pm$ to $G$ is isomorphic to $\mathscr{L}^{\otimes\kappa}$,
  where $\A^0$ means the identity component of $\A$.
  \item
   Let $I^+_L$ be the set of orbits $I^+_L:=I^+/L$. Then the reduced special fiber of $\Pm$ has a stratification indexed by $I^+_L$. This stratification is preserved by the action of $H$, and the induced action of $H$ on the set of strata is determined by the action of $H$ on $I^+_L.$
\end{enumerate}
\label{model for AV}
\end{Thm}
\begin{sss}\label{Pt}
  Let us discuss $\Pm$, which appears in Theorem \ref{model for AV}.
  For each cone $\sigma\in \Sigma$, we define the affine scheme $U_\sigma := \spec R[\sigma^\vee \cap \tilde{M}]$,
  where  $\sigma^\vee:= {\rm Hom}_{\rm monoid}( \sigma, \R_{\geq 0})$ and
  we identify $\tilde{m}=(m,k)\in \tilde{M}$ with $t^k X^m\in K[M]$.
  Then we can define $\Pt$ by glueing  these  $U_\sigma$ together as in \cite[1.13]{Kun}.
  In particular, we obtain the toroidal embedding
  $T_\eta =\spec k[M]\hookrightarrow \Pt$ as in \emph{loc.cit}.
  This $\Pt$ is called the \emph{toroidal compactification} of $T_\eta = \spec K[M]
  $ over $R$ associated with $\Sigma$.
  Further, the cone $\sigma_T$ induces the embedding $T_\eta \hookrightarrow T=U_{\sigma_T}=\spec R[M]$.
  It implies that the troidal embedding $T_\eta \hookrightarrow \Pt$
  extends to a $T$-equivariant embedding $T\hookrightarrow \Pt$.
   The special fiber of $\Pt$ is a reduced divisor with strict normal crossings on $\Pt$ and has a stratification indexed by $I^+$.

\emph{
If $(G, \mathscr{L}, \mathscr{M})\in \DEG$ is maximally dagenerated, then the above $\Pm$ of Theorem \ref{model for AV}
satisfies $\Pm_\f \cong \Pt_\f / L$.
}
Then, this $\Pt$ is also called relatively complete model as in \cite{CM_1972__24_3_239_0}.
In general, the above $\Pm$ is constructed by taking a
contraction product  $\tilde{G}\times^T\Pt$, which we do not use
in this paper. See \cite[\S3.6]{Kun} for the details.
\end{sss}
\begin{sss}\label{kulikov for Abelian}
  In \cite[Theorem 5.1.6]{Halle2017MotivicZF}, they proved this $\Pm$ is a Kulikov model of $A$
  (cf. \cite[Corollary 2.8]{overkamp2021}).
\end{sss}

\begin{sss}\label{action}
  For the tuple $(M,L, \phi, a,b):=\For(F((G, \mathscr{L}, \mathscr{M})))$,
  $b$ gives the injective homomorphism $\tilde{b}: L\to N=M^\vee$ defined by $\tilde{b}(l)=b(l,-)$.
  We identify $L$ with $\tilde{b}(L)$. That is, we regard $L$ as the sublattice of $N$.
  As we see before, $\Gamma$ act on $\tilde{N}$ as follows:
    $$S_{(l,h)}((n,s))=(n\circ h+s\tb(l), s)$$
    In particular,
      $$S_{(l,h)}((n,1))=(n\circ h+\tb(l), 1)$$
\end{sss}

\begin{sss}
K\"{u}nnemann proved the existence of the cone decomposition $\Sigma$
which satisfies the assumption of Theorem \ref{model for AV} as follows:
\end{sss}
\begin{Prop}[{\cite[Proposition 3.3 and Theorem 4.7]{Kun}}]
We set the tuple  $(G, \mathscr{L}, \mathscr{M})\in\DEG$ ,
and assume that the finite group $H$ acts on this object.
 Let $(M,L, \phi, a,b):=\For(F((G, \mathscr{L}, \mathscr{M}))).$
 After taking a base change along $f: S'\to S$ as in \eqref{base change} if necessary,
 there exists a smooth rational polyhedral cone
 decomposition $\Sigma := \{\sigma_{\alpha}\}_{\alpha\in I}$
 which has the properties (a)-(d) listed in
 Theorem \ref{model for AV}. \label{decomp1}
\end{Prop}
\begin{sss}\label{sketch of subdivision}
Now we recall K\"{u}nnemann's proof of the above Proposition \ref{decomp1}.
Please refer to \emph{loc.cit.} for more details. We consider the function $\varphi : \sC \to \R$ defined by
$$\tilde{n}\mapsto \min_{l\in L} \chi (l, \tilde{n}),$$
where $\chi (l, \tilde{n})$ means $\chi ((l,{\rm Id}), \tilde{n})$.
This $\varphi$ gives the decomposition $\Sigma=\{\sigma_\alpha\}$ defined as follows:
Let $\alpha=\{\alpha_i\}$ be a finite set of $L$. For such an $\alpha$, we consider
$$\sigma_\alpha:=\{\tilde{n} \in \sC \ |\ \varphi (\tilde{n})=\chi(\alpha_i, \tilde{n}) \ {}^\forall \alpha_i\in \alpha\}.$$
Here, for two finite set $\alpha$ and $\beta$ of $L$,
we define $\alpha \sim \beta$  by $\sigma_\alpha=\sigma_\beta$.
It is an equivalence relation.
Let $I$ be the set of equivalence classes of finite set of $L$.
Then $\Sigma:=\{\sigma_\alpha\}_{\alpha \in I}$ is a desired one.

Now $\varphi$ is a $1$-twisted polarization function associated with this $\Sigma$ as in \cite[Proposition 3.2]{Kun}.
In particular, it holds that $S_{(l,h)}(\sigma_\alpha)=\sigma_{h(\alpha)-l}$.
Now we consider the cone
$$\sigma_{\{0\}}=\{\tilde{n} \in \sC \ |\ \varphi(\tilde{n})=\chi (0,\tilde{n})=0\}.$$
It is clear that $\sC= \bigcup S_l (\sigma_{\{0\}})$.

\emph{First step} :
For this cone $\sigma_{\{0\}}$, we can subdivide it and obtain an $H$-invariant finite cone decomposition $\{\tau_\beta\}$ of $\sigma_{\{0\}}$
 such that each
cone $\tau_\beta$ is a simplex and the stabilizer of $\tau_\beta$ in $H$ acts trivially on $\tau_\beta$.
Further we can subdivide the whole $\Sigma$ by transporting
the above subdivision on $\sigma_{\{0\}}$ via $L$-action on $\sC$ and
 obtain an $H$-invariant cone decomposition $\{\tau_\alpha\}$ of $\sC$.
In addition, we can modify the polarization function $\varphi$ and obtain a $1$-twisted polarization function for
this subdivision $\{\tau_\alpha\}$ after replacing $K$ by a finite extension.

\emph{Second step} :
We choose a system $\{\tau_1,...,\tau_n\}$ of representatives for the action of $\Gamma$ on
the decomposition $\{\tau_\alpha\}$.
According to \cite[I.2, proof of Theorem 11]{kempf1973toroidal},
\emph{for any subdivision $\Sigma_i$ of each $\tau_i$}, there is a subdivision of the subdivision $\Sigma_i$
such that  it has a $\kappa$-twisted polarization function on $\tau_i$ for sufficiently large $\kappa\in \N$.
In the same way as above, we can extend these subdivisions to the whole via $L$-action.
Further, we can modify the polarization function on $\sC$ and obtain a $\kappa$-twisted polarization function for this subdivision $\Sigma'$
after replacing $K$ by a finite extension.
Hence, we consider a subdivision that satisfies (c), (d) to obtain a subdivision that satisfies (a), (c), (d).

\emph{Third step} :
We choose a system $\{\tau_1,...,\tau_n\}$ of representatives for the action of $\Gamma$ on
the decomposition $\Sigma'$.
By using the semistable reduction theorem \cite[II.2, proof of Theorem11]{kempf1973toroidal}, we can subdivide each $\tau_i$ so that
the resulting decomposition $\Sigma''$ is smooth.
In addition, we can obtain a $\kappa'$-twisted polarization function for this subdivision $\Sigma''$
after replacing $K$ by a finite extension.
Hence, the desired decomposition is constructed.
$\square$
\end{sss}

\begin{sss}
Let $B$ be a topological space endowed with a simplicial complex structure.
We denote by $\Sigma := \{\sigma_\alpha\}_{\alpha \in I}$ the set of all faces of $B$.
Let $\sigma^\circ$ be the relative open set of $\sigma \in \Sigma$.
We define the open star $\Star (\sigma_\alpha)$ of $\sigma_\alpha \in \Sigma$ as follows:
 $$\Star (\sigma_\alpha):= \bigcup_{\beta \succ \alpha} \sigma_\beta^\circ ,$$
 where $\beta \succ \alpha$ means that $\sigma_\alpha$ is a face of $\sigma_\beta$.
 Then $\Star(\sigma)$ is a open set of $B$. In particular, $\{\Star(\sigma_\alpha)\}_{\alpha \in I}$ is a open cover of $B$.
\end{sss}
\begin{sss}\label{dual intersection1}
  The decomposition $\Sigma:= \{\sigma_{\alpha}\}_{\alpha\in I}$ of $\sC$ as Theorem \ref{model for AV} gives
  the smooth rational polyhedral decomposition $\ol{\Sigma}$ in $N_\R$ obtained by intersectiong the cones in $\Sigma$ with $N_\R \times \{ 1\}$.
  Let $\overline{\sigma}_\alpha \in \ol{\Sigma}$ be the intersection of $\sigma_\alpha$ with $N_\R \times \{ 1\}$.
  Then this decomposition $\ol{\Sigma}= \{ \overline{\sigma}_\alpha \}_{\alpha \in I}$ gives a simplicial complex structure to $N_\R$. Moreover the dual intersection complex $\Delta (\Pt) $ of $\Pt_k$ coincides with $\ol{\Sigma}$ as we see in (\ref{Pt}).
  Theorem \ref{model for AV} implies that the dual intersection complex $\Delta(\Pm)$ of $\Pm_k$ has
  the simplicial complex structure of $\ol{\Sigma}/L:=\{ \overline{\sigma}_\alpha \}_{\alpha \in I^+_L}$.
\end{sss}
\begin{sss}
  To make it easier to see the covering map, which is the key to this paper
  and which we will look at later, we refine Proposition \ref{decomp1}.
\end{sss}

\begin{Lem}\label{fixed locus}
  Let $F$ be the fixed locus of $H$-action on $N_\R/L$.
  Then $N_\R/L$ has a simplicial complex structure such that any 0-vertex of $N_\R/L$ is included in $N_\Q/L$ and $F$ is a compact sub simplicial complex.

\end{Lem}
\begin{proof}
It follows from Proposition \ref{decomp1} that $N_\R/L$ has a simplicial complex structure such that
the canonical projection $N_\R\twoheadrightarrow N_\R /L$ is a simplicial map
and any 0-simplex in $N_\R/L$ is included in $N_\Q/L$.
We show that $F$ is compact and has a simplicial complex structure such that any 0-simplex of $F$ is
included in $N_\Q/L$. Note that $F$ might not be connected.
Then the assertion follows from
\cite[2.12 Addendum]{RS82}.

From now on, we prove the claim.
By definition, $F$ is denoted as follows: $$F=\bigcup_{h\in H\setminus \{0\}} F_h,$$
where $F_h:=\{x\in N_\R/L\ |\ x=h(x) \}$.
Under the setting as we considered in \eqref{how to act}, $H$ acts on $N_\R/L$
via
$H\to \left( {\rm GL}(L) \ltimes L \right) \cap \left( {\rm GL}(N)\ltimes N\right)$.
In other words, $h$ is determined by $\tilde{h}\in \left( {\rm GL}(L) \ltimes L \right) \cap \left( {\rm GL}(N)\ltimes N\right)$ via the canonical projection $N_\R\to N_\R/L$.
Here, we identify $h$ with $\tilde{h}$.
In particular, denote $h:N_\R\to N_\R$ by $h(y)=Ay+b$, where $A\in {\rm GL}(L)\cap {\rm GL}(N)$ and $b\in L(\subset N)$.
Set $g:N_\R\to N_\R$ as $g(y):=h(y)-y$.
For any $y_1, y_2\in N_\R$, $g(y_1+y_2)=g(y_1)+Ay_2-y_2$ holds. It implies that
 $g(y)\in L$ is equivalent to $g(y+a)\in L$ for some $a\in L$.
 Taka a basis $\{l_i\}$ of $L$ in $N_\R$ and set a fundamental domain $D$ of $N_\R/L$ as follows: $$D:=\sum_{i=1}^{\dim N} [0,1]\cdot l_i\subset N_\R.$$
 Then it holds that
 $$F'_h:=\{y\in D\ |\ g(y)\in L\} \twoheadrightarrow F_h$$
by the canonical projection $N_\R\to N_\R/L$.
Since $L$ is discrete in $N_\R$ and $g(D)$ is compact, $V:=g(D)\cap L$ is a finite set.
Here, we can denote $F'_h$ by $$F'_h=\bigcup_{v\in V}F'_{h,v},$$
where $F'_{h,v}:=\{ y\in D\ |\ g(y)=v\}$.
Since $g$ is an integral affine map (or $g\in {\rm Hom}(N,N)\ltimes N$), $F'_{h,v}$ is a closed set of
a subaffine space in $N_\R$ with rational slopes for the coordinates of $N_\R$ containing some point in $N_\Q$.
It implies that $F'_{h,v}$ has a simplicial structure such that each vertex of $F'_{h,v}$ is in $N_\Q$.
In particular, the inclusion $F'_{h,v} \to N_\R$ is a piecewise linear map.
Since the canonical projection $N_\R\twoheadrightarrow N_\R /L$ is a simplicial map,
the composition $F'_{h,v} \to N_\R \twoheadrightarrow N_\R/L$ is also a piecewise linear map.
By \cite[2.14 Theorem]{RS82}, there are subdivisions of $N_\R /L$ and $F'_{h,v}$ such that
$F'_{h,v} \to N_\R /L$ is a simplicial map and
any 0-simplex of the subdivision of $N_\R /L$ is included in $N_\Q /L$.
Hence, the image $F_{h,v}$ of $F'_{h,v}$ has
a simplicial structure induced by the simplicial map $F'_{h,v} \to N_\R /L$.
Here, any intersection (as a cell complex) between two cell complexes is also a cell complex.
Since $F'_{h,v}$ is an intersection of $D$ and
some subaffine space in $N_\R$ with rational slopes intersecting $N_\Q$, for any $F_{h,v}$ and $F_{h',v'}$,
it holds that any 0-cell of the cell complex of the intersection  is included in $N_\Q/L$.
Hence, the union of $F_{h,v}$ and $F_{h',v'}$ is a cell complex such that all 0-cells are in $N_\Q/L$ by gluing together along the intersection cell complex.
It is well-known that any cell complex can be subdivided to a simplicial complex without introducing any new vertices.
That is, the union of $F_{h,v}$ and $F_{h',v'}$ has a simplicial complex structure such that all 0-simplexes are in $N_\Q/L$.
Since $F$ is a finite union of simplicial complexes,
more precisely $$F=\bigcup_{h\in H\setminus \{e\}}\bigcup_{v\in V} F_{h,v},$$
then it follows inductively that $F$ has a simplicial structure such that any vertex of $F$ is included in $N_\Q /L$.
Besides, since $F_{h,v}$ is compact, so is $F$.
Hence, the claim follows.
\end{proof}
\begin{Lem}\label{key subdivision}
  Let $F$ be the fixed locus of $H$-action on $N_\R/L$.
  Let $\tilde{F}$  be the inverse image of the fixed locus $F$ by the quotient map $N_\R\to N_\R/L$.
  After taking a base change along $f: S'\to S$ as in \eqref{base change} if necessary, there exists a smooth rational polyhedral cone decomposition
  $\Sigma = \{\sigma_{\alpha}\}_{\alpha\in I}$
  which has not only the properties (a)-(d) listed in Theorem \ref{model for AV}
  but also the following (e)-(g).
  \begin{enumerate}
  \item[(e)]
    For all $l\in L\backslash \{ 0\}$ and $\alpha\in I_+$, we have
  $$\Star(\overline{\sigma}_\alpha)\cap S_{(l,\mathrm{Id})}(\Star(\overline{\sigma}_\alpha))=\emptyset.$$
  \item[(f)]  $F$ has a simplicial structure $\mathscr{F}$ such that $F$ is a subcomplex of the complex $\overline{\Sigma}/ L$ as appeared in \eqref{dual intersection1}. In particular,
  for any simplex $\tau$ of the induced simplicial structure $\tilde{\mathscr{F}} $ on $\tilde{F}$,
  a cone generated by $(\tau, 1)\subset \tilde{N}_\R:=N_\R \times \R$
  corresponds to some index in $I$.
We denote by $I_{\rm sing}\subset  I_+$ the set of indices corresponding to  $\tilde{\mathscr{F}}$.

  \item[(g)]   For all $\gamma \in \Gamma \backslash \{ 0\}$ and $\alpha \in I_+\setminus  I_{\rm sing}$, we have
  $$\Star(\overline{\sigma}_\alpha)\cap S_\gamma(\Star(\overline{\sigma}_\alpha))=\emptyset.$$
\end{enumerate}
\end{Lem}
\begin{proof}
It follows from Proposition \ref{decomp1} that there is a smooth rational polyhedral cone decomposition $\Sigma$ which satisfies the conditions (a)-(d) after replacing $K$ by a finite extension.
Then we refine $\Sigma$ to obtain a desired decomposition as follows:
In the second step of (\ref{sketch of subdivision}),
we consider a subdivision that satisfies (e), (f), (g).
Note that each stabilizer of $H$ on  each $\tau\in \Sigma$ acts trivially on the cone $\tau$ and $L$ acts on $\tau$ by transporting via $\tilde{b}(L)$. Then
 it is easily verified that such subdivisions exist by Lemma \ref{fixed locus}.
Afterward, we apply the third step of (\ref{sketch of subdivision})
to this decomposition.
Then the resulting decomposition is a desired one.
\end{proof}
\begin{Ex}
If $H=\{\pm 1\}$, then $\tilde{F}=\frac{1}{2}L$ and $F=\frac{1}{2}L/L$.
In particular, it holds that $|F|=2^{\dim N}$. Further, $|F/H|=2^{\dim N}$ follows.
\end{Ex}

\begin{sss}
  For the rest of this section, we assume that \emph{the residue field $k$
  of $R$ is of characteristic $p \neq 2$},
We set that
$H=\{\pm 1\}$
   and the action of $H$ on $M$   is determined by $-1: m\mapsto -m$.
  In particular,  $H=\{\pm 1\}$ also acts on $N=M^\vee$ by $-1: n\mapsto -n$.
\end{sss}
\begin{sss}\label{Kulikov for Kummer}
  Let $\Pm$ be the projective model of $A$ and $\A$ be the N\'{e}ron model of $A$ as Theorem \ref{model for AV}.
  For an abelian variety $Z$, we denote by $Z[2]$ the 2-torsion of $Z$, that is the kernel of the morphism $[2]: Z \to Z$ defined by $x\mapsto 2x$.
  After replacing $K$ by finite extension, we may assume that $A[2]$ is constant over $K$ without loss of generality.
  Overkamp proved this $\A[2]$ coincides with the fixed locus of the action of $H$ on $\Pm$
  when $A$ is of $2$-dimensional \cite[Theorem 3.7]{overkamp2021}. Then the action of $H=\{\pm 1\}$ on $\Pm$ extends to the blow-up $\cXt:= \mathrm{Bl}_{\A[2]}\Pm$ along the closed subscheme $\A[2]$. Hence we obtain $\cX:=\cXt/H$.
  Let $X$ be the Kummer surface associated with $A$.
  Overkamp proves this $\cX$ is a Kulikov model of $X$ \cite[Theorem 3.12]{overkamp2021}.
\end{sss}
\begin{sss}\label{dual intersection2}
We fix the same notation as (\ref{dual intersection1}) and (\ref{Kulikov for Kummer}).
  The dual intersection complex $\Delta ( \cXt)$ of $\cXt_k$ has the same stratification as the dual intersection complex $\Delta (\Pm)$ of $\Pm_k$. Indeed, Overkamp proved that the special fiber $\cXt$ is $ \mathrm{Bl}_{\A_k[2]}\Pm_k$ \cite[Lemma 3.10]{overkamp2021}
  and $\A_k[2]$ is a finite set lying on top dimensional strata of $\Pm_k$
  \cite[Lemma 3.6]{overkamp2021}.
  We can also check the latter by using Lemma \ref{key subdivision}.
  Hence, the blow-up along $\A_k[2]$ does not change the dual intersection complex. It implies that $\Delta (\Pm) \cong \Delta (\cXt)$ as simplicial complexes.

  We denote by $I^+_\Gamma$ the set of orbits $I^+_\Gamma :=I^+/\Gamma$.
  Theorem \ref{model for AV} says that $H$ acts on $\Delta (\Pm) \cong \Delta (\cXt)$ preserving  the simplicial complex structure.
  It implies that the map $\Delta (\cXt) \twoheadrightarrow \Delta (\cX)$ is double branched cover as simplicial complexes.
The dual intersection complex $\Delta ( \cX)$ of $\cX_k$ has a stratification indexed by $I^+_\Gamma$.
In particular, $\Delta ( \cX)$   has the simplicial complex structure of $\ol{\Sigma}/\Gamma :=\{ \overline{\sigma}_\alpha \}_{\alpha \in I^+_\Gamma}$.
\end{sss}

\section{Non-Archimedean SYZ Fibration}

In this section, we introduce some important results from \cite{nicaise_xu_yu_2019}.
\begin{sss}
For the rest of this paper, we assume that \emph{the characteristic of the residue field $k$ is $0$}.
\end{sss}
\begin{Def}\label{ess}

Let $X$ be a Calabi-Yau variety over $K$ and $\omega$ be a volume form on $X$.
Then we can define the \emph{weight function}
$$ \wtf : X^\an \to \R\cup \{\infty\}.$$
Please refer to \cite[\S 4.5]{MuNi} for details.
The \emph{essential skelton} $\Sk (X)$ of $X$ is the subset of $X^\an$ consisting of points where
 $\wtf$ reaches its minimal value.
Since $X$ is Calabi-Yau, $\omega$ is uniquely determined up to a scalar multiple.
Multiplying $\omega$ with a scalar changes the weight function by a constant.
Therefore,  $\Sk (X)$ depends only on $X$ not on $\omega$.
\end{Def}
\begin{sss}
Let $X$ be a smooth connected $K$-variety and
let $\cX$ be an snc-model of $X$ over $S$.
The dual intersection complex $\Delta (\cX)$ of $\cX_k$ is canonically embedded into $X^\an$ \cite[Theorem 3.1]{BFJ}.
We denote by $\Sk(\cX)$ its image of $\Delta (\cX)$.
$\Sk(\cX)$ is called the \emph{Berkovich skelton} of $\cX$ and has the simplicial structure induced by $\Delta (\cX)$.
If $X$ is a Calabi-Yau variety over $K$, then the essential skelton $\Sk (X)$ as in Definition \ref{ess}
is canonically homeomorphic to the subcomplex of  $\Sk (\cX)$.
    Since $\cX$ is snc,
      it follows from \cite[3.3.3]{ddae55dd7a7343b590c46f80664c6a79} that
     the image of this embedding is exactly the essential skeleton $\Sk(X)$.
     In particular, we give a simplicial complex structure to $\Sk (X)$ by the one of $\Sk(\cX)$.
\end{sss}

\begin{Def}\label{B-retraction}
  Let $X$ be a smooth connected $K$-variety and
  let $\cX$ be an snc-model of $X$ over $S$.
  We assume that $X^\an = \cX_\ber$.
  In particular, if $\cX$ is projective over $S$,
  then $X$ is projective over $K$ and
   this assumption holds.
Here, we construct the \emph{Berkovich retraction} associated with an snc-model $\cX$ of $X$ in accordance with \cite[(2.4)]{nicaise_xu_yu_2019} (or \cite[\S 3]{BFJ}).

   Let $x$ be a point in $X^{\an}$ and let $\red_{\cX}(x)$ be its reduction on $\cX_k$ as we saw in Definition \ref{def of reduction}.
  We denote by $Z$ the smallest stratum containing $\red_{\cX}(x)=\xi$. 
  Then 
  it determines a unique face $\sigma$ of the dual intersection complex $\Delta(\cX)$.
  Let $D_1,...,D_r$ be the irreducible components of $\cX_k$ that contain $Z$, and let $N_1,...,N_r$ be their multiplicities in $\cX_k$.
   Then $D_1,...,D_r$ correspond to the vertices $v_1,...,v_r$ of $\sigma$.
  We choose a positive integer $m$ such that $mD_i$ is Cartier at the point $\red_{\cX}(x)$ for every $i$, and we choose a local equation $f_i=0$ for $mD_i$ at $\red_{\cX}(x)$.
  Then $\rho_{\cX}(x)$ is defined as the point of the simplex $\sigma$ with barycentric coordinates
   $$\alpha=\frac{1}{m}(-N_1\log|f_1(x)|,\ldots,-N_r\log|f_r(x)| )$$ with respect to the vertices $(v_1,...,v_r)$.
    The image $\rho_{\cX}(x)$ of $x$ corresponds to the monomial point represented by
    $(\cX,(D_1,...,D_r),\xi)$ and the tuple $$\frac{1}{m}(-\log|f_1(x)|,\ldots,-\log|f_r(x)|),$$ in the terminology of \cite[2.4.5]{MuNi}
    via the embedding of $\Delta(\cX)$ into $X^{\an}$.
    We can easily verify that this definition does not depend on the choices of $m$ and the local equations $f_i$
   and check that  $\rho_{\cX}$ is continuous, and that it is a retraction onto
     the skelton $\Sk(\cX)=\Delta(\cX)$.
\end{Def}

\begin{Def}\label{SYZ}
Let $X$ be a Calabi-Yau variety over $K$. If an snc-model $\cX$ of $X$ is
a good minimal dlt-model of $X$ with a technical assumption
 as in \cite[(1.11)]{nicaise_xu_yu_2019}.
 Then we call the map $\rho_{\cX}\colon X^{\an}\to \Sk(X)$ constructed in Definition
 \ref{B-retraction} the \emph{non-Archimedean SYZ fibration} associated with $\cX$.
\end{Def}
\begin{sss}
  \label{int str}
We note that, even though the subspace $\Sk(X)$ of $X^{\an}$ only depends on $X$,
  the simplicial complex structure on $\Sk(X)$
  and the non-Archimedean SYZ fibration $\rho_{\cX}: X^{\an}\to \Sk(X)$ depend on
  the choice of the good minimal dlt-model $\cX$.
In \cite[\S3.2]{MuNi}, the authors discussed
the canonical piecewise integral affine structure of $\Sk (X)$
and revealed that this piecewise integral affine structure coincides
with the one induced by $\Delta (\cX)$.
In other words, the piecewise integral affine structure induced by $\Delta (\cX)$ does not depend on
  the choice of the good minimal dlt-model $\cX$.
However, this is closer to the topological structure than to the integral affine structure.
We note that we focus on the integral affine structure (more precisely, IAMS structure) in this paper.
\end{sss}

\begin{sss}

Let $T$ be a split algebraic $K$-torus of dimension $n$ with its character group $M$.
We denote by $N=M^{\vee}$ the dual module of $M$.  We define the \emph{tropicalization map} $\rho_T : T^{\an}\to N_{\R}$ of $T$ by
$$ T^\an \ni x\mapsto (m \mapsto -\log |m(x)|)\in M_\R^\vee =N_\R.$$
Then $\rho_T$ is continuous, and its fibers are (not necessarily strictly) $K$-affinoid tori.
 Further, the tropicalization map $\rho_T$ has a canonical continuous section $s\colon N_{\R}\to T^{\an}$
 that sends each $n\in N_{\R}$ to the Gauss point of the
 affinoid torus $\rho^{-1}_{T}(n)$. The image of $s$ is called the \emph{canonical skeleton} of $T$, and denoted by $\Delta(T)$.
 The map $s$ induces a homeomorphism $N_{\R}\to \Delta(T)$.
 We identify $\Delta(T)$ with $N_{\R}$ via this homeomorphism.
\end{sss}
\begin{Def}\label{afd torus fib} Let $Y$ be a $K$-analytic space, let $B$ be a topological space and
  let $f\colon Y\to B$ be a continuous map.
  Then  $f$ is called an $n$-dimensional \emph{affinoid torus  fibration}
  if there is a open covering $\{U_i\}$  of $B$ such that,   for each $U_i$, there is an open subset $V_i$ of $N_{\R} \cong \R^n$ and a commutative diagram
$$\xymatrix{
f^{-1}(U_i)\ar[r] \ar[d]_{f} & \rho_T^{-1}(V_i) \ar[d]^{\rho_T}
\\ U_i \ar[r] & V_i\ar@{}[lu]|{\circlearrowright}
}$$
where the upper horizontal map is an isomorphism of $K$-analytic spaces and the lower horizontal map is a homeomorphism.
\end{Def}

\begin{sss} If $f\colon Y\to B$ is an affinoid torus  fibration, then $f$ induces
  an integral affine structure on the base $B$
  as follows: For each open set $U$ in $B$ as
  in Definition \ref{afd torus fib}, we consider an invertible analytic function $h$ on $f^{-1}(U)$.
  Then
  the absolute value of $h$ is constant along the fibers of $f$ \cite[\S4.1, Lemma 1]{Kontsevich2006}.
   Hence $h$ implies a continuous function
$|h|\colon U\to \R_{>0}$ by taking $|h(b)|$ as $|h(y)|$ for some $y\in f^{-1}(b)$.
 We can define the integral affine functions on $U$ as
 the functions of the form $-\log|h|$.
 If $U$ is connected,
 then we can identify
 the ring of integral affine functions on $U$
  with the ring of polynomial functions of degree $1$ with $\Z$-coefficients
 on $V\subset N_{\R}$ so that this construction indeed defines an integral affine structure
 on $B$ via
  the homeomorphism $U\to V$ \cite[\S4.1, Theorem 1]{Kontsevich2006}.
  More precisely, in \emph{loc.cit.}, they considered affine functions
   whose coefficients are in $\R$,
 rather than $\Z$.
However, that's because they allowed the base field $K$ to be a general nontrivial valued field.
Under the condition that $K$ is a discrete-valued field as in our setting, we can obtain affine functions
 whose coefficients are in $\Z$ as above. That is, we can give
  the integral affine structure to $B$ in this way.
   We call it \emph{non-Archimedean SYZ Picture}.
\end{sss}

\section{Affine Structures for  Degenerations of Kummer Surfaces}
\subsection{Non-Archimedean SYZ Picture}
\begin{sss}
First, we prepare two settings, one for general use and one for Kummer surfaces.
If it is too complicated,  it is enough to just consider the latter setting \eqref{setting},
which is a special case of the former \eqref{general setting}.
\end{sss}
\begin{sss}[general setting]\label{general setting}
Let $A$ be an Abelian variety over $K$ and
 $\A$ be the N\'{e}ron model of $A$.
After taking a  base change along $f: S'\to S$ as in \eqref{base change} if necessary, there is a triple $(G,\mathscr{L}, \sM)\in \DEG$
such that $A=G_\eta$ and $G=\A^0$
by the semiabelian reduction \cite[Expos\'{e} I, Th\'{e}or\`{e}me 6.1]{semiabelian}.
In addition, we may assume that a finite group $H$ acts on $(G,\mathscr{L}, \sM)$ such that
the fixed locus of $H$ on $A$ is constant $K$-group scheme by taking a further base change along $f: S'\to S$ as above,
without loss of generality.
 We assume that $G$ is maximally degenerated, which is the same as $\A$ being.
For the tuple $(M,L,\phi, a,b)=\For(F((G, \mathscr{L}, \mathscr{M})))$, there is a decomposition $\Sigma$ as
Lem \ref{key subdivision}
after taking a  base change along $f: S'\to S$ as above.
In particular, the decomposition $\Sigma$ is $\Gamma=L\rtimes H$-admissible.

  Let $\Pt$ be the toroidal compactification of $T =\spec K[M]$ over $R$ associated with $\Sigma$ as constructed in (\ref{Pt}) and
$\Pm$ be the projective model of $A$ as Theorem \ref{model for AV}.
This $\Pt$ is an snc model of $T$.
$\Pm$ is a Kulikov model of $A$ as we see in (\ref{kulikov for Abelian}).
By definition, this Kulikov model $\Pm$ is a good minimal dlt model with a technical assumption
as in \cite[(2.3)]{nicaise_xu_yu_2019}.
Hence, it follows that $\Sk (A)=\Sk (\Pm)$.
Further, we replace $\mathscr{L}$ by $\mathscr{L}^{\otimes \kappa}$ so that $\mathscr{L}$ extends to the ample
line bundle $\mathscr{L}_\Pm$ on $\Pm$. Since $\M$ is trivial in our setting,
there is no need to consider $\M$ in particular.
Since it holds that $T^\an =\Pt_\ber$ and $A^\an=\Pm_\ber$,
we can define the Berkovich retractions for these snc-models $\Pt$ and $\Pm$.
We denote by $\rho_{\Pt}$ (resp. $\rho_\Pm$)
the Berkovich retraction associated with $\Pt$ (resp. $\Pm$)
as in Definition \ref{B-retraction}.
In particular, $\rho_\Pm$ is a non-Archimedean SYZ fibration.
Let $\rho_T$ be the tropicalization map of $T$.
\end{sss}
\begin{sss}[setting for Kummer surfaces]\label{setting}
Let $A$ be an Abelian surface over $K$ and $X$ be the Kummer surface associated with $A$.
We denote by $\A$ the N\'{e}ron model of $A$.
After taking a  base change along $f: S'\to S$ as in \eqref{base change} if necessary, there is a $(G,\mathscr{L}, \sM)\in \DEG$
such that $A=G_\eta$ and $G=\A^0$
by the semiabelian reduction \cite[Expos\'{e} I, Th\'{e}or\`{e}me 6.1]{semiabelian}.
In addition, we may assume that
 the group $H=\{\pm 1\}$ acts on $(G,\mathscr{L}, \sM)$ so that
 the $K$-group scheme $A[2]$ is constant by taking a further base change  along $f: S'\to S$ as above, without loss of generality.
 We assume that $G$ is maximally degenerated, which is the same as $\A$ being.
For the tuple $(M,L,\phi, a,b)=\For(F((G, \mathscr{L}, \mathscr{M})))$, there is a decomposition $\Sigma$ as
Lem \ref{key subdivision}
after taking a  base change along $f: S'\to S$ as above.
In particular, the decomposition $\Sigma$ is $\Gamma=L\rtimes H$-admissible.

  Let $\Pt$ be the toroidal compactification of $T =\spec K[M]$ over $R$ associated with $\Sigma$ as constructed in (\ref{Pt}) and
$\Pm$ be the projective model of $A$ as Theorem \ref{model for AV}.
This $\Pt$ is an snc model of $T$.
$\Pm$ is a Kulikov model of $A$ as we see in (\ref{kulikov for Abelian}).
Further, we replace $\mathscr{L}$ by $\mathscr{L}^{\otimes \kappa}$ so that $\mathscr{L}$ extends to the ample
line bundle $\mathscr{L}_\Pm$ on $\Pm$. Since $\M$ is trivial in our setting,
there is no need to consider $\M$ in particular.
We denote by $\cX$ the Kulikov model of $X$ associated with $\Sigma$ as in (\ref{Kulikov for Kummer}).
By definition, these Kulikov models $\Pm$ and $\cX$ are good minimal dlt models with a technical assumption
as in \cite[(2.3)]{nicaise_xu_yu_2019}.
Hence, it holds that $\Sk (A)=\Sk (\Pm)$ and $\Sk (X)=\Sk (\cX)$.
In addition, we note that $T^\an =\Pt_\ber$, $A^\an=\Pm_\ber$ and $X^\an=\cX_\ber$.
Hence, we can define the Berkovich retractions for these snc-models $\Pt,\Pm$ and $\cX$.
We denote by $\rho_{\Pt}$ (resp. $\rho_\Pm$, $\rho_\cX$ )
 the Berkovich retraction associated with $\Pt$ (resp. $\Pm$, $\cX$) as in Definition \ref{B-retraction}.
 In particular, $\rho_\Pm$ and $\rho_\cX$ are non-Archimedean SYZ fibrations.
 Let $\rho_T$ be the tropicalization map of $T$.
\end{sss}
\begin{Rem}
As we can see, the setting \eqref{general setting} is a generalization of \eqref{setting}.
Under the setting \eqref{general setting}, we
consider a general dimensional abelian variety with an action of a general finite group.
However, we do not consider the quotient $X$ under this setting \eqref{general setting} in this paper.
It is because
we are not sure that an analog of what Overkamp proved on Kummer surfaces in \cite{overkamp2021} also works.
\end{Rem}
\begin{Prop}\label{SYZ for torus}

  Under the setting as in \eqref{general setting}, the Berkovich retraction
  $\rho_{\Pt}$ of $\Pt$ is equal to the tropicalization map $\rho_{T}$.
  In particular, $\rho_{\Pt}$ is
an
  affinoid torus fibration.
\end{Prop}
\begin{proof}
We set $d:=\dim N$.
  The decomposition $\Sigma$ gives the smooth rational polyhedral decomposition $\overline{\Sigma}$ in $N_\R$ obtained by intersectiong the cones in $\Sigma$ with $N_\R \times \{ 1\}$.
  As we saw in (\ref{dual intersection1}),
  the Berkovich skelton $\Sk (\Pt)$ coincides with $N_\R$.
  Moreover, simplicial structure of $\Sk (\Pt)$ coincides with $\ol{\Sigma}$.
  Let $\sigma\in \Sigma$ be the smallest cone containing $\rho_T(x) \in N_\R\cong N_\R\times \{1\}$.

  We set $\sigma =\R_{\geq 0} \tilde{n}_0+\cdots +\R_{\geq 0} \tilde{n}_s$, where $\tilde{n}_i=(n_i,1)$.
  We extend these elements to a $\Z$-basis $\tilde{n}_0,...,\tilde{n}_{d}$ of $\tilde{N_\R}$.
  Let $\tilde{m}_i=(m_i,r_i)\in \tilde{M}$ be the dual basis of $\tilde{M}$.
  We may assume that $$\rho_T(x) =\sum_{i=0}^s a_i  \tilde{n}_i=:\tilde{n}=(n,1)\in N_\R \times \{ 1\} \cong N_\R,$$
  where $\sum a_i=1$ and $a_i>0$ for all $0\leq i\leq s$.

  We set $A_\sigma = R[\tilde{M}\cap \sigma^\vee] \cong R[Y_0,...,Y_s,Y_{s+1}^{\pm} ,...,Y_d^{\pm}]/(Y_0 \cdots Y_s -t)$, where $Y_i:= t^{r_i}X^{m _i}$.
  Then $U_\sigma := \spec A_\sigma \subset \Pt$.

  It follows that $-\log |Y_j(x)|= \langle\tilde{m_j}, \tilde{n}\rangle = \langle\tilde{m_j}, \sum a_i  \tilde{n_i} \rangle =a_j$ for $x\in \rho_T^{-1}(n)$ and for all $0\leq j\leq d$.
  Therefore $\red_\cX(x)$ coincides with the generic point $\xi_\sigma$ of
  the toric stratum $D_\sigma$ corresponding to $\sigma$.
  Moreover, each irreducible component $D_i$ of $\Pt_0$ that contains $D_\sigma$ corresponds to each one dimensional face
  $\tau_i =\R_{\geq 0} \tilde{n}_i$ of $\sigma$.
  Therefore, it follows that $$\rho_\Pt (x) = \sum_{i=0}^s a_i  \tilde{n}_i = \rho_T(x).$$
\end{proof}
\begin{Cor}\label{SYZ for 2-torus}
  Under the setting as in \eqref{setting}, the Berkovich retraction
  $\rho_{\Pt}$ of $\Pt$ is equal to the tropicalization map $\rho_{T}$.
  In particular, $\rho_{\Pt}$ is
  a $2$ dimensional
  affinoid torus fibration.

\end{Cor}
\begin{proof}
It follows by exactly the same argument as above Proposition \ref{SYZ for torus}.
\end{proof}
\begin{sss}\label{ret cover}
  Under the setting as in Definition \ref{B-retraction},
  let $\rho_\cX\colon X^{\an}\twoheadrightarrow \Sk (\cX)\subset X^\an$ be the Berkovich retraction, where the simplicial structure of
  $\Sk (\cX)$ is given by a decomposition $\ol{\Sigma}/L=\{ \overline{\sigma}_\alpha \}_{\alpha \in I^+_L}$ under the notation as in \eqref{dual intersection2}.
  Since the retraction $\rho_\cX$ is continuous, the inverse image $\rho_\cX^{-1}(\Star (\sigma_\alpha))$ is an open set.
  In particular, it holds that $$X^\an = \bigcup_{\alpha \in I}\rho_\cX^{-1}(\Star (\sigma_\alpha)).$$
  We call this covering \emph{ the retraction covering of $X^\an$ associated with $\cX$.}
  In other words, we can regard taking an snc-model of $X$ as taking a retraction covering of $X^\an$.
  To be precise, the stratification of the formal completion $\cX_\f$ gives the retraction covering.
  We note that $\rho_\cX^{-1}(\Star (\sigma_\alpha))=\red_\cX^{-1}(D_\alpha)$, where $D_\alpha$ is the
  scheme-theoretic intersection of the irreducible components corresponding to $1$-dimensional faces of $\sigma_\alpha$.
  Let $\xi_\alpha$ be a stratum of $\cX_k$
  corresponding to $\sigma_\alpha$.
  Then $D_\alpha=\ol{\{\xi_\alpha\}}$.
\end{sss}

\begin{sss}\label{simplicial covering}

  For the decomposition $\Sigma = \{\sigma_\alpha\}_{\alpha \in I} $ as in \eqref{general setting},
the Berkovich skelton $\Sk (\Pt)$ is described as follows:
  $$\Sk(\Pt)=\bigcup_{\alpha \in I^+}\ol{\sigma}_\alpha \cong N_\R \cong N_\R\times \{1\},$$
  where $\ol{\sigma}_\alpha:= \sigma_\alpha \cap (N_\R\times\{ 1\} )$ as in \eqref{dual intersection1}.
  Theorem 3.3 implies that $\Gamma = L\rtimes H$ acts on $\Sk (\Pt)$  as follows:
  $$S_{(l,h)}((n,1))=(n\circ h+\tb (l), 1).$$
Moreover,
  $\Sk(\Pm)=\bigcup_{\alpha \in I^+_L}\ol{\sigma}_\alpha$ (resp.
  $B:=\bigcup_{\alpha \in I^+_\Gamma}\ol{\sigma}_\alpha$)
  is isomorphic to $\Sk (\Pt)/ L$  (resp. $\Sk (\Pt) / \Gamma$) as simplicial complex.
  By Lemma \ref{key subdivision}, the morphism $\Sk (\Pt) \to \Sk (\Pm)$ is an unbranched cover
  such that its fundamental group is isomorphic to $L$, and
   the morphism $\Sk (\Pm) \to B$ is a branched double cover.
   Under the more concrete condition \eqref{setting}, the
 ramification locus of this
   morphism $\Sk (\Pm) \to B=\Sk (\cX)$
   is $Z:=\frac{1}{2}L/\Gamma$.
   In particular, $Z$ consists of 4 points.
 \end{sss}

   \begin{sss}\label{equiv}
   Under the setting as in \eqref{general setting},
   the action of $\Gamma$ on $\Pt$ induces $\Gamma$-action on $T^\an$ via the Raynaud generic fiber.
   In partucular, the reduction map $\red_\Pt$ and the Berkovich retraction map $\rho_\Pt$ are $\Gamma$-equivariant.
   That is, it holds that  $\rho_\Pt(\gamma \cdot x)=S_\gamma(\rho_\Pt(x))$ for all $x\in T^\an$ and $\gamma \in \Gamma$.
   Further, we can also verify that the Berkovich retraction  $\rho_\Pm$ of $\Pm$ is $H$-equivaliant, similarly.

\end{sss}

\begin{Lem}\label{commute}
  Under the setting \eqref{general setting},
  the following diagram commutes.

  $$\xymatrix{
  T^\an \ar[r]^{/L}
  \ar[d]_{\rho_\Pt} & A^\an \ar[d]_{\rho_\Pm}
  \\ \Sk(\Pt)
   \ar[r]^{/L} &  \Sk(\Pm) \ar@{}[lu]|{\circlearrowright}
  }$$
\end{Lem}
\begin{proof}
Since $G$ is maximally degenerated, it holds that $\Pm_\f \cong \Pt_\f /L$ as in \eqref{Pt}.
In particular, we obtain the morphism $f : \Pt_\f \to \Pm_\f$.
Then $f_\ber :T^\an\to A^\an$ is the morphism appearing in the above diagram.
Let
 $g: \Sk(\Pt)\to\Sk(\Pm)$
be the morphism appearing in the above diagram, similarly.
Here, the proof is completed by showing the commutativity
$\rho_\Pm \circ f_\ber =g\circ\rho_\Pt$.
By definition, the image $\rho_\Pt(x)$ of $x\in T^\an$ is determined by the point $\xi=\red_\Pt(x)$ coresponding to the cone $\sigma_\xi \in \Sigma$, the irreducible components $D_1,...,D_r$
containing $\xi$
and the barycentric coordinates $(v_1,...,v_r)$ with respect to the vertices corresponding to these $D_i$,
where each $D_i$ corresponds to the $1$-dimensional face $\sigma_{\alpha_i}$ of the cone $\sigma$ for some $\alpha_i\in I^1$.
Note that each $f(D_i)$ is an irreducible component since $f_\ber$ is a covering map.
Then the image $\rho_\Pm(f_\ber(x))$ is determined by the point $f(\xi) =\red_\Pm(f_\ber(x))$, the irreducible components $f(D_1),...,f(D_r)$
and the barycentric coordinates $(v_1,...,v_r)$ with respect to the vertices corresponding to these $f(D_i)$,
where each $f(D_i)$ corresponds to $1$-dimensional cone $\sigma_{\ol{\alpha}_i}$
for some $\ol{\alpha}_i\in I^+_L$ as in Theorem \ref{model for AV}, where $\alpha_i \in I^1$
is the one above.
On the other hand,  $g(\rho_\Pt(x))$ is determined by the simplex $g(\ol{\sigma}) \in \ol{\Sigma}/L$ and the
barycentric coordinates $(v_1,...,v_r)$ with respect to the vertices $g(\ol{\sigma}_{\alpha_i})$,
where $\alpha_i\in  I^1$
is the one above.
Here, the retraction $\rho_\Pt: T^\an\to\Sk (\Pt)$ is $L$-equivaliant as we see in \eqref{equiv}.
Hence we obtain $\rho_\Pm(f_\ber(x))=g(\rho_\Pt(x))$.
That is, the above diagram commutes.
\end{proof}
\begin{Prop}\label{commute for Kummer}
  Under the setting \eqref{setting}, let $\pi$ be the blow up $\pi : \mathrm{Bl}_{A[2]}A \to A$.
  The following diagram commutes.

  $$\xymatrix{
  & (\mathrm{Bl}_{A[2]}A)^\an  \ar[d]_{\pi^\an} \ar[rd]^{H\backslash}& \\
  T^\an \ar[r]^{/L}
  \ar[d]_{\rho_\Pt} & A^\an \ar[d]_{\rho_\Pm}  & X^\an \ar[d]_{\rho_\cX}
  \\ \Sk(\Pt)
  \ar@/_15pt/[rr]_{/\Gamma}
   \ar[r]^{/L} &  \Sk(\Pm) \ar@{}[lu]|{\circlearrowright} \ar[r]^{ H\backslash }& \Sk(\cX) \ar@{}[luu]|{\circlearrowright}
  }$$
\end{Prop}
\begin{proof}
It follows by the same argument as above Lemma \ref{commute}
that the left part of the above diagram commutes.
Hence it is enough to show that the right part of the above diagram commutes.
We set $\cXt:= \mathrm{Bl}_{\A[2]}\Pm$ as in (\ref{Kulikov for Kummer}).
This $\cXt$ is an snc model of $\mathrm{Bl}_{A[2]}A$. We denote by $\rho_\cXt$ the Berkovich retraction.
Since $\Sk (\Pm) = \Sk (\cXt)$
as we see in (\ref{dual intersection2}),
it holds that
$\rho_\cXt = \rho_\Pm \circ\pi^\an$.
Since $\pi$ is the blow-up along the fixed locus of $H$,
the blow-up $\pi$ is $H$-equivaliant. In particular,
$H$-equivaliant retraction $\rho_\Pm$ implies that $\rho_\cXt$ is $H$-equivaliant.
After that, we can check the commutativity directly by representing the two images concretely as in the proof of Lemma \ref{commute}.
Hence, the right part of the above diagram commutes.
\end{proof}
\begin{Prop}[cf.{\cite[Proposition 3.8]{nicaise_xu_yu_2019}}]\label{covering1}
Under the setting \eqref{general setting}, the morphism $T^\an \to A^\an$  is an unbranched cover.
Moreover the open sets of the form $\rho_\Pm^{-1}(\Star (\ol{\sigma}_\alpha))$
for any $\alpha \in I^+$ are evenly covered neighborhoods.
In particular, $\rho_\Pm$ is
an
affinoid torus fibration.
\end{Prop}
\begin{proof}

By the property (e) of Lemma \ref{key subdivision},
$\Star (\ol{\sigma}_\alpha)\subset \Sk (\Pm)$ is an evenly covered
neighborhood with respect to $\Sk (\Pt)\to \Sk (\Pm)$,
where we identify $\Star (\ol{\sigma}_\alpha)\subset \Sk (\Pm)$ with one of the sheets
$\wt{\Star} (\ol{\sigma}_\alpha) \subset \Sk (\Pt)$.
For each $l\in L\setminus \{ 0\}$, the following diagram holds.
$$\xymatrix{
\rho_\Pt^{-1}(\wt{\Star}  (\ol{\sigma}_\alpha)) \ar[r]^{\simeq}_{l}
\ar[d]_{\rho_\Pt} &l\cdot\rho_\Pt^{-1}(\wt{\Star}  (\ol{\sigma}_\alpha))\ar[d]_{\rho_\Pt}
\\ \Star (\ol{\sigma}_\alpha)
 \ar[r]_{S_{l}}^{\simeq} &  S_{l}(\Star (\ol{\sigma}_\alpha))
}$$
In particular,
 the upper horizontal map is an isomorphism of $K$-analytic spaces
 and the lower horizontal map is a homeomorphism.
The property (e) of Lemma \ref{key subdivision} says that
$ \wt{\Star}  (\ol{\sigma}_\alpha)\cap S_{l}(\wt{\Star}  (\ol{\sigma}_\alpha))=\emptyset$.
It implies that
$\rho_\Pt^{-1}(\wt{\Star}  (\ol{\sigma}_\alpha)) \cap
l\cdot\rho_\Pt^{-1}(\wt{\Star}  (\ol{\sigma}_\alpha))=\emptyset$.
By Lemma \ref{commute},
 we obtain $\rho_\Pt^{-1}(\wt{\Star}  (\ol{\sigma}_\alpha))\cong \rho_\Pm^{-1}(\Star (\ol{\sigma}_\alpha))$.
That is, we can identify $\rho_\Pm^{-1}(\Star (\ol{\sigma}_\alpha))$ with one of the sheets
 $\rho_\Pt^{-1}(\wt{\Star}  (\ol{\sigma}_\alpha))$.
Hence,  $\rho_\Pm^{-1}(\Star (\ol{\sigma}_\alpha))$ is an evenly covered neighborhoods.
By Proposition \ref{SYZ for torus}, $\rho_\Pt=\rho_T$ follows.
It implies the last assertion.
\end{proof}
\begin{Cor}
Under the setting \eqref{setting}, the morphism $T^\an \to A^\an$  is an unbranched cover.
Moreover the open sets of the form $\rho_\Pm^{-1}(\Star (\ol{\sigma}_\alpha))$
for any $\alpha \in I^+$ are evenly covered neighborhoods.
In particular, $\rho_\Pm$ is
a 2-dimensional
affinoid torus fibration.
\end{Cor}
\begin{proof}
  It follows by the same argument as above Proposition \ref{covering1}.
\end{proof}
\begin{sss}
  In \cite[Proposition 3.8]{nicaise_xu_yu_2019}, they used the decomposition $\Sigma$
  which is constructed in Proposition \ref{model for AV} and proved that
  the Berkovich retraction $\rho_\Pm$ does not depend
  on the choice of such decomposition.
  On the other hand, the reason why we adopted the decomposition which is constructed in
  Lemma \ref{key subdivision}
  is to show directly that $\rho_\Pm$ is an affinoid torus fibration
  by looking at the covering map concretely.

\end{sss}
\begin{Cor}\label{covering2}
  Under the setting \eqref{setting},
  the morphism $T^\an \setminus \rho_T^{-1}(\frac{1}{2}L) \to X^\an \setminus \rho_\cX^{-1}(Z)$  is an unbranched cover.
  Moreover the open sets of the form $\rho_\cX^{-1}(\Star (\ol{\sigma}_\alpha))$
  for any $\alpha \in I^+\setminus I_{\rm sing}$ are evenly covered neighborhoods.
  In particular, the restriction of $\rho_\cX $ to the open set $ X^\an \setminus \rho_\cX^{-1}(Z)$ is a 2-dimensional affinoid torus fibration.
\end{Cor}
\begin{proof}
The morphism $(\mathrm{Bl}_{A[2]}A)^\an \to X^\an$ as in Proposition \ref{commute for Kummer} induces the morphism $$A^\an \setminus
\rho_\Pm^{-1}(\frac{1}{2}L/L) \to X^\an \setminus \rho_\cX^{-1}(Z)$$
by restricting to the open set which is isomorphic to $A^\an \setminus
\rho_\Pm^{-1}(\frac{1}{2}L/L)$. By composing with $T^\an\to A^\an$, we consider the morphism
$$T^\an\setminus \rho_T^{-1} (\frac{1}{2}L) \to  X^\an \setminus \rho_\cX^{-1}(Z).$$
By the property (f) of Lemma \ref{key subdivision}, the above exceptional part $\frac{1}{2}L$ corresponds to $ I_{\rm sing}$.
By the property (g) of Lemma \ref{key subdivision},
$\Star (\ol{\sigma}_\alpha)\subset \Sk (\cX)$ is an evenly covered
neighborhood with respect to $\Sk (\Pt)\to \Sk (\cX)$ for all $\alpha \in I^+ \setminus I_{\rm sing}$.
Hence, this morphism $T^\an \setminus \rho_T^{-1}(\frac{1}{2}L) \to X^\an \setminus \rho_\cX^{-1}(Z)$
is an unbranched cover.
Moreover, we obtain
the latter assertion by using
 Proposition \ref{covering1}.
\end{proof}

\begin{Prop}[cf.{\cite[(3.6), Proposition 3.8]{nicaise_xu_yu_2019}}]\label{SYZ for av}
  Under the setting \eqref{general setting},
   the induced integral affine structure on $\Sk(A)$ by $\rho_\Pm$ coincides with the quotient structure on $N_\R /L$.
\end{Prop}
\begin{proof}

It follows from  \eqref{kulikov for Abelian} that $\Sk(A)=\Sk (\Pm)$.
By Proposition \ref{covering1}, the non-Archimedean SYZ fibration $\rho_\Pm$ is an affinoid torus fibration. Hence this fibration
$\rho_\Pm$ induces the integral affine structure on $\Sk (A)$.
Then  the following commutative diagram
$$\xymatrix{
T^\an \ar[r]^{/L} \ar[d]_{\rho_T} & A^\an \ar[d]^{\rho_\Pm}
\\ N_\R \ar[r]_{/L} &  \Sk(A) \ar@{}[lu]|{\circlearrowright}
}$$
gives the morphism $N_\R \to \Sk (A)$ between integral affine manifolds.
In particular, this morphism is defined by taking the quotient of $N_\R$ by the lattice $\tb :L\hookrightarrow N_\R$.
Hence, this finishes the proof.
\end{proof}
\begin{Cor}\label{rdc obs}
  Let $T^2=N_\R/L$ be the integral affine manifold  constructed in Proposition \ref{SYZ for av},
  and let $\mathcal{T}_{T^2}$ be the local system on $T^2$ of lattices of tangent vectors.
  Then, the radiance obstruction $c_{T^2} \in {\rm H}^1(T^2,\mathcal{T}_{T^2})$
  (cf.\cite{GH},
  \cite{article})
  coincides with $\tb \in  {\rm Hom}(L, N) \subset {\rm Hom}(L, N_\R)$ via
  the canonical isomorphism
  $ {\rm H}^1(T^2,\mathcal{T}_{T^2})\cong {\rm Hom}(L, N_\R).$
\end{Cor}
\begin{proof}
  It directly follows from Proposition \ref{SYZ for av}.
\end{proof}
\begin{sss}
In \cite[Theorem 6.1]{nicaise_xu_yu_2019}, they proved that
for each maximally degenerating projective Calabi-Yau variety $X$ over $K$ and
any good minimal dlt-model $\cX$ over $S$,
the singular locus $Z$ of the essential skeleton $\Sk(X)$ with the IAMS structure induced by $\Sk(\cX)$
is contained in the union of the faces of codimension $\geq 2$ in $\Sk(\cX)$.
In particular, the singular locus is of codimension $\geq 2$.
 Further,
 in
 \emph{loc.cit.},
 they proved that the \emph{piecewise} integral affine structure of $\Sk(X)$
 induced by this IAMS structure of $\Sk(X)$ does not depend on the choice of such dlt-models.

As we state in \eqref{int str}, however, what is called piecewise integral structure
is closer to the topological structure than to the integral affine structure.
In other words, the IAMS structure of $\Sk(X)$
induced by $\Sk(\cX)$ \emph{does} depend on the choice of such dlt-models.
In general, it is difficult to describe its IAMS structure explicitly,
but in the case of Kummer surfaces, it can be described as follows:
\end{sss}
\begin{Thm}[The Affine Structure via non-Archimedean SYZ Picture]\label{nAside}
  Under the setting \eqref{setting}, the restriction  of
  the non-Archimedean SYZ fibration $\rho_{\cX} : X^\an \to
  \Sk(X)$ to the open set $X^\an \setminus \rho_\cX^{-1}(Z)$ is a $2$-dimensional
   affinoid torus fibration.
  Moreover, the integral affine structure on $\Sk (X) \setminus Z$ induced by
  $\rho_{\cX}$ coincides with the restriction of the quotient structure
  on $N_\R / \Gamma$, where $\Gamma=L\rtimes H$.

\end{Thm}
\begin{proof}
It follows from \eqref{Kulikov for Kummer} that $\Sk(X)=\Sk (\cX)$.
By Corollary \ref{covering2}, $\rho_{\cX} |_{X^\an \setminus \rho_\cX^{-1}(Z)}$ is an affinoid torus fibration.
The following commutative diagram
$$\xymatrix{
T^\an \setminus \rho_T^{-1}(\frac{1}{2}L) \ar[r] \ar[d]_{\rho_T} & X^\an \setminus \rho_\cX^{-1}(Z) \ar[d]^{\rho_\cX}
\\ N_\R \setminus \frac{1}{2}L \ar[r]_{/\Gamma} &  \Sk(X)\setminus Z \ar@{}[lu]|{\circlearrowright}
}$$
 gives the unbranched cover $ N_\R \setminus \frac{1}{2}L \to  \Sk(X)\setminus Z$.
In the same manner as above Proposition \ref{SYZ for av}, we obtain the isomorphism
$$  \Sk(X)\setminus Z \cong (N_\R \setminus \frac{1}{2}L)/\Gamma = (N_\R/\Gamma) \setminus \{4 pts\} $$
as an integral affine manifold.
\end{proof}
\begin{Cor}\label{rad obs for Kummer}
Let $S^2=N_\R/\Gamma$ be the IAMS constructed in Theorem \ref{nAside} and
let $\mathcal{T}_{S^2\setminus Z}$ be the local system on $S^2\setminus Z$ of lattices of tangent vectors.
We denote by $\iota :S^2\setminus Z\to S^2$ the natural inclusion.
Then  the radiance obstruction $c_{S^2} \in {\rm H}^1(S^2,\iota_* \mathcal{T}_{S^2\setminus Z})$
  coincides with $\frac{1}{2}\tb \in {\rm Hom}(L, N_\R)$
  via the isomorphism ${\rm Hom}(L, N_\R) \cong  {\rm H}^1(T^2,\mathcal{T}_{T^2})\cong {\rm H}^1(S^2,\iota_* \mathcal{T}_{S^2\setminus Z})$ induced by the
  quotient morphism
  $T^2 \to S^2$ between these IAMS.
  Further, the radiance obstruction $c_{S^2}$ is contained in ${\rm Hom}(L,N)$.
\end{Cor}
\begin{proof}
  Tsutsui proved that the quotient morphism $q: T^2\to S^2$ induces the isomorphism
  $q_* :{\rm H}^1(T^2,\mathcal{T}_{T^2})\cong {\rm H}^1(S^2,\iota_* \mathcal{T}_{S^2\setminus Z})$ such
  that $c_{S^2}=\frac{1}{2}c_{T^2}$ holds in his unpublished work\cite{Tsurvey}.
      Hence, the first assertion directly follows
      from Theorem \ref{nAside}, Corollary \ref{rdc obs} and the above Tsutsui's work.

      On the other hand,
      Overkamp proved that the map $b:L\times M\to \Z$ as in \eqref{setting} takes only even values
      \cite[Proposition 3.5]{overkamp2021}. Hence, $\tb: L \to N$ also takes only even values.
      It implies that $c_{S^2}=\frac{1}{2}\tb\in {\rm Hom}(L,N)$.
\end{proof}
\begin{Rem}\label{indep of pol}
  Under the setting \eqref{setting}, these IAMS are uniquely determined by $M$, $L$ and $b$.
  Hence, these IAMS do not depend on the polarization $\phi$.
\end{Rem}

\begin{sss}
In \cite{GO}, which we wrote with Odaka after this paper, we study $K$-trivial finite quotients of abelian varieties by generalizing the discussion of Kummer surfaces which we have seen in this paper.
For the rest of this subsection, we consider the generalized case in advance of that.
\end{sss}

 \begin{Thm}[{\cite[Proof of Corollary 3.3]{GO}}]\label{SNC model for fqav}
 Consider an arbitrary abelian variety $A$ over $K$
 of dimension $g$
 with an action of a finite group $H$ as appeared in \eqref{general setting}.
 Assume that $H$ acts trivially on the canonical bundle $\omega_A$ on $A$ so that the canonical bundle $\omega_{A/H}$ on $A/H$  is trivial.
 Then,
 for the $H$-equivariant SNC model $\Pm$ as appeared in Theorem \ref{model for AV}, the pair
 $(\Pm/H,(\Pm /H)_k)$ is qdlt in the sense of \cite{dFKX}.
 Further,
 we assume that for any nontrivial $h\in H$, the fixed locus of its action on $\Sk (A)$ is 0-dimensional.
 Then $(\Pm/H,(\Pm /H)_k)$ is dlt.

 \end{Thm}
 \begin{proof}
 We write the irreducible decomposition of
 $\Pm_k$ as $\cup_i E_i$. We want to show that
 for any $h$ which is not the identity $e$,
 $h$ does not fix any $E_i$ pointwise.
 Suppose the contrary and take a general point of
 $x\in E_i$. $\Pm$ is smooth over  $R$
 at $x$, where $R$ is the DVR of $K$.
 We take local coordinates
 $(x_1,\cdots,x_g)$ of $x\in E_i$ which we extend to $h$-invariant coordinates
 around
 $x\in \Pm$. Then
 $(x_1,\cdots,x_g,t)$ is a $H$-invariant
 local coordinates of $x\in \Pm$, which contradicts with
 nontriviality of $h$.
 Hence $(\Pm/H)_k$ is reduced.
 In particular, it implies that the quotient $(\Pm/H,(\Pm/H)_k)$ is qdlt.

 Suppose there is $h(\neq e)\in H$
 which preserves a strata $Z$ of $\Pm_k$
 (a log canonical center of
 $(\Pm,\Pm_k)$) pointwise. Then
 the strata of the dual complex $\Sk(A)$
 which corresponds to $Z$ is fixed by $h$, hence
 contradicts with our last assumption.
 It implies that the quotient
 $(\Pm/H,(\Pm/H)_k)$ is dlt.
 \end{proof}

 \begin{Rem}
   Theorem \ref{SNC model for fqav} partially extends a result by Overkamp
   in the Kummer surfaces case
   cf., \cite[\S 2, \S 3]{overkamp2021}).
 \end{Rem}

 \noindent We are now in a position to generalize Theorem \ref{nAside}.
 \begin{sss}[setting for K-trivial finite quotients of abelian varieties]\label{setting fqav}

 Let $A$ be a $g$-dimensional abelian variety over $K$,
 $H$ be a group satisfying the whole conditions as appeared in Theorem \ref{SNC model for fqav}, and
 $X$ be the quotient of the abelian variety $A$ by the group $H$.
 We denote by $\A$ the N\'{e}ron model of $A$.
 After taking a  base change along $f: S'\to S$ as in \eqref{base change} if necessary, there is a $(G,\mathscr{L}, \sM)\in \DEG$
 such that $A=G_\eta$ and $G=\A^0$
 by the semiabelian reduction \cite[Expos\'{e} I, Th\'{e}or\`{e}me 6.1]{semiabelian}.
  Assume that $G$ is maximally degenerated,  which is the same as $\A$ being.
 For the tuple $(M,L,\phi, a,b)=\For(F((G, \mathscr{L}, \mathscr{M})))$, there is a decomposition $\Sigma$ as
 Lem \ref{key subdivision}
 after taking a  base change along $f: S'\to S$ as above.
 In particular, the decomposition $\Sigma$ is $\Gamma=L\rtimes H$-admissible.

   Let $\Pt$ be the toroidal compactification of $T =\spec K[M]$ over $R$ associated with $\Sigma$ as constructed in (\ref{Pt}) and
 $\Pm$ be the projective model of $A$ as Theorem \ref{model for AV}.
 This $\Pt$ is an SNC model of $T$.
 $\Pm$ is a Kulikov model of $A$ as we see in (\ref{kulikov for Abelian}).
 Further, we replace $\mathscr{L}$ by $\mathscr{L}^{\otimes \kappa}$ so that $\mathscr{L}$ extends to the ample
 line bundle $\mathscr{L}_\Pm$ on $\Pm$. Since $\M$ is trivial in our setting,
 there is no need to consider $\M$ in particular.
 We denote by $\cX$ the dlt model of $X$ associated with $\Sigma$ as in Theorem \ref{SNC model for fqav}.
 That is, $\cX:=\Pm /H$.
 By definition, these Kulikov models $\Pm$ and $\cX$ are good minimal dlt models with a technical assumption
 as in \cite[(2.3)]{nicaise_xu_yu_2019}.
 Hence, it holds that $\Sk (A)=\Sk (\Pm)$ and $\Sk (X)=\Sk (\cX)$.
 In addition, we note that $T^\an =\Pt_\ber$, $A^\an=\Pm_\ber$, and $X^\an=\cX_\ber$.
 Hence, we can define the Berkovich retractions for these models $\Pt,\Pm$, and $\cX$.
 We denote by $\rho_{\Pt}$ (resp., $\rho_\Pm$, $\rho_\cX$ )
  the Berkovich retraction associated with $\Pt$ (resp., $\Pm$, $\cX$) as in Definition \ref{B-retraction}.
  In particular, $\rho_\Pm$ and $\rho_\cX$ are non-Archimedean SYZ fibrations.
  Let $\rho_T$ be the tropicalization map of $T$.
  Here, we denote by $Z\subset \Sk (\cX)$ the ramification locus of $\Sk (\Pm) \to \Sk (\cX)$.
 \end{sss}

 \begin{Thm}[NA SYZ picture for K-trivial finite quotients of abelian varieties]\label{nA SYZ for fqav}
   Under the setting \eqref{setting fqav}, the restriction  of
   the non-Archimedean SYZ fibration $\rho_{\cX} : X^\an \to
   \Sk(X)$ to the open set $X^\an \setminus \rho_\cX^{-1}(Z)$ is an
    affinoid torus fibration.
   In partucular, the integral affine structure on $\Sk (X) \setminus Z$ induced by
   $\rho_{\cX}$ coincides with the restriction of the quotient structure
   on $N_\R / \Gamma$, where $\Gamma=L\rtimes H$.
   Moreover, the skelton $\Sk (X)$ is an IAMS, that is ${\rm codim} Z\geq 2$.
 \end{Thm}
 \begin{proof}
   In a similar way as Proposition \ref{commute for Kummer},
   the following diagram commutes.
      $$\xymatrix{
      T^\an \ar[r]^{/L}
      \ar[d]_{\rho_\Pt} & A^\an \ar[r]^-{H\backslash} \ar[d]_{\rho_\Pm}  & X^\an \ar[d]_{\rho_\cX}
      \\ \Sk(\Pt)
      \ar@/_15pt/[rr]_{/\Gamma}
       \ar[r]^{/L} &  \Sk(\Pm) \ar@{}[lu]|{\circlearrowright} \ar[r]^{ H\backslash }& \Sk(\cX) \ar@{}[lu]|{\circlearrowright}
      }$$
 Hence, it induces an isomorphism $\Sk (X) \backslash Z \cong ( N_\R / \Gamma) \backslash Z$
 as an integral affine manifold by the same discussion as Theorem \ref{nAside}.

 To finish the proof, we show  ${\rm codim} Z\geq 2$.
 Since $X$ is $K$-trivial, the ramification divisor $R$ of the finite morphism $f:A\to X$ vanishes.
 If $D$ is a fixed prime divisor on $A$ for some $h\in H$,
 then $D$ is a ramification divisor on $A$.
 Indeed, when we set $\xi$ and $\xi'$ as the generic point of $D$ and $f(D)$,
  the dimension of the finite morphism $\str_{X,\xi'} \to \str_{A,\xi}$ between two DVR's is given by the stabilizer of $H$.
 In particular, $\dim_{\str_{X,\xi'}} (\str_{A,\xi})\geq 2$.
 On the other hand, $\dim_{\str_{X,\xi'}} (\str_{A,\xi})$ is equal to the value of an uniformizing parameter of $\str_{X,\xi'}$ for the discrete valuation on $\str_{A,\xi}$. This is nothing but the ramification index of $D$.
 Hence, $D$ is a ramification divisor.
 From now on, we show that $A$ has a fixed divisor for some $h\in H$ if  ${\rm codim} Z = 1$.
 As in \eqref{how to act}, any action $h$ on $A$ can lift to an action on the split torus $T(=\spec K[M])$.
 In particular, we may assume $h \in {\rm GL}(N) \cap {\rm GL} (L)$, where $N=M^\vee$.
 Then this action descends to the skeleton $\Sk (\Pm) \cong N_\R /L$ via the canonical projection $N_\R \to N_\R/L$.
 By construction, any action $h\in H$ on $\Sk (\Pm)$
 is given in this way.
 If ${\rm codim} Z =1$, then some $h$ fixes some 1 codimensional subspace in $N_\R$.
 Fix such an $h$.
 Here, for the simplicial decomposition of $N_\R$ as in Lemma \ref{key subdivision}, the stabilizer of $H$ on each simplex is trivial.
 It implies that $h$ has $g-1$ linear independent eigenvectors with eigenvalues 1. In particular, $h$ is diagonalizable.
 Further, $h\in  {\rm GL}(N)$ implies
 $\det h=\pm 1$.
 If $\det h= 1$, then $h$ must be trivial.
 Hence, $\det h =-1$.
 That is,
 $h$ is diagonalizable with eigenvalues $(-1,1,\dots, 1)$.
 Since $N=M^\vee$, the same holds for the action of $h$ on $M$.
 Then we can take eigenvectors of $h$ in $M$ since $h \in {\rm GL}(M)$.
 In particular,
 we take a primitive eigenvector $m\in M$ of $h$ with eigenvalue $-1$.
 Here, $z^m-1=z_1^{m_1}\cdots z_g^{m_g}-1$ is a prime element of $K[M]$.
 Indeed, we can take a basis of $M$ such that an element of the basis is $m\in M$
 since $M/m\Z$ is a free $\Z$-module and $0\to m\Z \to M\to M/m\Z\to 0$.
 That is, we may assume that $m=(1,0,\dots, 0)$ after taking some $B\in {\rm GL}(M)$.
 Then
 $z^m-1=z_1-1$ is a prime element of $K[M]$
 since
 $$K[M]/(z_1-1)\cong K[z_2^{\pm},\dots, z_g^{\pm}].$$
 Note that $m$ is an eigenvector of $h$ with eigenvalue $-1$. It implies that
 the prime divisor on $T$ defined by $z^m-1=0$ is invariant for $h$.
 Note that we can take
  an affinoid domain $V$ of $T^{\rm an}$ such that
  the restriction of $T^{\rm an} \to A^{\rm an}$ to $V$ is an isomorphism and
 the interior of $V$ intersects the closed analytic space defined by $z^m-1=0$.
 It implies that $A$ has a prime divisor locally defined by  $z^m-1=0$.
 In particular, the prime divisor on $A$ is invariant for $h$.
 It is a contradiction. Hence ${\rm codim} Z \geq 2$ follows.
 \end{proof}

\subsection{Gromov-Hausdorff limit Picture}
\begin{sss}\label{setting2}
  In this section, we also consider the same situation as (\ref{setting}).
  Furthermore, we assume that $k=\C$ and $(G,\mathscr{L})$ is principally polarized (that is, $\phi:L\to M$ is an isomorphism).
We set $B(l_i,l_j):= b(l_i,\phi (l_j))$, where $\{ l_i\}$ is a basis of $L$.
By definition, $B:L\times L\to \Z$ is a symmetric positive definite quadric form.
We set $\Delta:=\{t\in \C \ |\ |t|<1 \}$ and $\Delta^*:=\Delta\setminus \{0\}$.
For given $(G,\mathscr{L})$, we assume that  $G(\C)\twoheadrightarrow \Delta$, where $G(\C)$ is the
 analytification of $G$ in the sense of complex analytic space.
 For abbreviation, we write $G$ instead of $G(\C)$.
\end{sss}
\begin{sss}
  We recall the Gromov-Hausdorff limit (cf. \cite{BBI}).
  We can define the Gromov-Hausdorff distance $d_{\rm GH}(X,Y)$ between two metric spaces $X$ and $Y$.
  It is known that this distance $d_{\rm GH}$ is a metric function
  on the set $\mathbb{M}$ consisting of the isometry classes of compact metric spaces.
  In Gromov's celebrated paper \cite{gromov1981structures},  he proved that
  the subset $\mathcal{M}$ of $\mathbb{M}$ consisting of the isometry classes of compact Riemannian manifolds
  with Ricci curvature bounded below and diameter bounded above
  is relatively compact with respect to the Gromov-Hausdorff distance.
  It is known as Gromov's compactness theorem.
  That is, the closure $\ol{\mathcal{M}}^{\rm GH}$ of $\mathcal{M}$ in $\mathbb{M}$
  is compact.
Hence we can define a notion of convergence for sequences in $\mathcal{M}$,
called Gromov-Hausdorff convergence.
In particular, for any sequence of Ricci flat
manifolds, we can take a convergent subsequence by rescaling the diameters to be $1$.
A compact metric space to which such a sequence converges is called a \emph{Gromov-Hausdorff limit} of the sequence.
\end{sss}

\begin{sss}
We discuss the existence of special Lagrangian fibrations
near maximally degenerated fibers (`large complex structure limit')
for K-trivial surfaces, as expected in the mirror symmetry context, essentially
after \cite{OO}. For instance, it is called `metric SYZ conjecture' in
\cite{YangLi}.

Before stating the statements for abelian surfaces and their quotients,
we recall the result proven in \cite{OO}. For simplicity of description,
we identify Riemannian metrics and induced distance.
\end{sss}

\begin{Thm}[{\cite[Chapter 4, especially pp.34-35, 46-49]{OO}}]\label{OO.review}
For any maximally degenerating
family of polarized K3 surfaces $(\mathcal{X}|_{\Delta^*},\mathcal{L}|_{\Delta^*})$
over $\Delta^*$, the following hold. Here, we denote the fiber over $t$ as $(X_t,L_t)$.
\begin{enumerate}
\item \label{k3.i} For any $t\in \Delta^*$ with $|t|\ll 1$,
there is a special Lagrangian fibration $X_t \to B_t$
with respect to the K\"ahler form $\omega_t$
of the Ricci-flat K\"ahler metric $g_{\rm KE}(X_t)$
with $[\omega_t]=c_1(\mathcal{L}_t)$
and the imaginary part ${\rm Im}(\Omega_t)$ of a non-zero element
$(0\neq ) \Omega_t\in H^0(X_t,\omega_{X_t})$.
Here, $B_t$ is homeomorphic to $S^2$.

We note that $\omega_t$ and ${\rm Im}(\Omega_t)$ induce
affine structures with singularities on $B_t$ as $\nabla_{A}(t)$  and $\nabla_{B}(t)$ respectively,
as well as its McLean metric $g_t$
(cf., e.g., \cite[\S 3]{Hit}, \cite[\S 1]{gross2013mirror}, \cite[\S 4.3]{OO}).

\item \label{k3.ii}
We assume $|t|\ll 1$.
We consider the obtained $S^2$ associated with an IAMS structure,
and the McLean metric
$(B_t,\nabla_A(t),\nabla_B(t),g_t)$ for $t\neq 0$. When $t\to 0$,
they converge in the natural sense,
without collapsing,
to another $S^2$
again with three additional structures
$(B_0,\nabla_A(0),\nabla_B(0),g_0)$.

In this terminology, the Gromov-Hausdorff limit of $(X_t,g_{\rm KE}(X_t))$ for $t\to 0$
coincides with $(B_0,g_0)$.
\end{enumerate}
\end{Thm}
\begin{Rem}
  Since hyperK\"ahler rotation is performed in the process of obtaining
  a special Lagrangian fibration, the complex dimension of $X_t$ must be $2$ (or even).
\end{Rem}
Below we discuss the case of abelian surfaces and their quotients.
In that case, we can apply similar methods as below
but more explicitly as \eqref{ab.iii} and \eqref{ab.iv} below.
The proof follows from essentially the same
method as \cite{OO} and is easier.
The details of the proof appears in \cite{GO}, which we wrote with Odaka after this paper.

\begin{Thm}[{\cite[Theorem 2.1]{GO}}]\label{GHside}
We take an arbitrary maximally degenerating
family of polarized abelian surfaces $(G|_{\Delta^*},\mathcal{L}|_{\Delta^*})$
over $\Delta^*$ with a fiber-preserving symplectic action of finite group
$H$ on $G|_{\Delta^*}$ together with linearization on
$\mathcal{L}|_{\Delta^*}$ (e.g., $H$ can be trivial or simple $\{\pm 1\}$-multiplication).
We denote the quotient by $H$ as
$(G'|_{\Delta^*},\mathcal{L}'|_{\Delta^*})\to \Delta^*$.
Then, the following hold:
\begin{enumerate}

\item \label{ab.i} For any $t\in \Delta^*$ with $|t|\ll 1$,
there is a special Lagrangian fibration $f_t\colon G_t \to B_t$
with respect to the K\"ahler form $\omega_t$
of the flat metric $g_{\rm KE}(G_t)$
with $[\omega_t]=c_1(\mathcal{L}_t)$
and the imaginary part ${\rm Im}(\Omega_t)$
of a non-zero element
$ \Omega_t\in H^0(G_t,\omega_{G_t})$.
Here, $B_t$ is a $2$-torus and so are all fibers of $f_t$.
Note that $\omega_t$ and ${\rm Im}(\Omega_t)$ again induce
affine structures on $B_t$ as $\nabla_{A}(t)$  and $\nabla_{B}(t)$ respectively,
as well as its (flat) McLean metric $g_t$.

Below, we assume $|t|\ll 1$.

\item \label{ab.ii}
We consider the obtained base associated with an integral affine structure and a flat metric
$(B_t,\nabla_A(t),\nabla_B(t),g_t)$ for $t\neq 0$. They converge to another
a $2$-torus with the same additional structures
$(B_0,\nabla_A(0),\nabla_B(0),g_0)$ in the natural sense, when $t\to 0$.

In this terminology, the Gromov-Hausdorff limit of $(G_t,g_{\rm KE}(G_t))$ for $t\to 0$
coincides with $(B_0,g_0)$.

\item \label{ab.iii}
The $H$-action on $G_t$ preserves the fibers of $f_t$.
Thus, there is a natural induced action of $H$ on $B_0$, which preserves the
three structures $\nabla_A(0),\nabla_B(0)$ and $g_0$.
The natural quotient of $f_t$ by $H$ denoted as
$f'_t\colon G'_t\to B'_t$ is again a special Lagrangian fibration
with respect to the descents of $\omega_t$
and $(0\neq ) \Omega_t\in H^0(G_t,\omega_{G_t})$.

\item \label{ab.iv}
If $\mathcal{L}|_{\Delta^*}$ is principally polarized and $H$ is trivial,
the integral point of $B_0$ with respect to the integral affine structure $\nabla_A(0)$
consists of only $1$ point, which automatically determines $\nabla_A(0)$.
The corresponding Gram matrix of $g_0$ is the matrix $(cB(l_i,l_j))$, where $c\in\R$ is
 a correction term to make the diameter to $1$.
Also the transition function of the integral basis of $\nabla_A(0)$ to that of $\nabla_B(0)$
is given by the same matrix $(cB(l_i,l_j))$.

\end{enumerate}
\end{Thm}

We are now in a position to define the Gromov-Hausdorff limit Picture.
\begin{Cor}\label{GHside for Kummer}
  Under the setting \eqref{setting2}, we use the same notation as above Theorem \ref{GHside}.
  Then the Gromov-Hausdorff limit of $(G'_t,g'_{\rm KE}(G'_t))$ for $t\to 0$
  coincides with the Gromov-Hausdorff limit $(B'_0,g'_0)$ of $(B'_t,g'_t)$ for $t\to 0$,
  where the metric $g'_{\rm KE}(G'_t)$ (resp. $g'_t$) on $G'_t$ (resp. $B'$) is induced by $g_{\rm KE}(G_t)$ (resp. $g_t$).
  Furthermore, the affine manifold $(B'_0,\nabla_B'(0))$ with singularities coincides with the quotient of the
   affine manifold $(B_0,\nabla_B(0))$
  by $H=\{\pm 1\}$, where the affine structure $\nabla_B'(0)$ with singularities is induced by $\nabla_B(0)$.
  In particular, we can regard the affine structure $\nabla_B'(0)$ with singularities as an IAMS structure by rescaling.
\end{Cor}
\begin{proof}
It follows from Theorem \ref{GHside} (3).
Since the affine structure $\nabla_B(0)$ of $B_0$ is detemined by
the matrix $B(l_i,l_j)$ up to scaling, the last assertion holds.
\end{proof}
\begin{sss}
  For the degenerateing family
  $(G'|_{\Delta^*},\mathcal{L}'|_{\Delta^*})$ as in Theorem \ref{GHside},
  we can give the IAMS structure $\nabla_{B}'(0)$ (up to scaling) to the
Gromov-Hausdorff limit $B'_0$ as above Corollary \ref{GHside for Kummer}.
We call it \emph{Gromov-Hausdorff limit Picture}.
For instance, it is also called \emph{Collapse Picture} in \cite{Kontsevich2006}.

On the other hand, for degenerating family of polarized abelian surfaces
$(G|_{\Delta^*},\mathcal{L}|_{\Delta^*})$ as in Theorem \ref{GHside}, we can also
give the integral affine structure $\nabla_{B}(0)$ (up to scaling) to the
Gromov-Hausdorff limit $B_0$ as Theorem \ref{GHside}.
Then we also call it Gromov-Hausdorff limit Picture.
\end{sss}

\begin{Thm}\label{Conj3 for as}
  Under the setting \eqref{setting2}, we use the same notation as Theorem \ref{GHside}.
Then the integral affine manifold induced by non-Archimedean SYZ Picture coincides with
  the integral affine manifold induced by the Gromov-Hausdorff limit Picture up to scaling.
  That is,
  $\Sk (\Pm)$ and $\nabla_B(0)$ give the same integral affine structure (up to scaling) to
   the 2-torus $T^2\cong \R^2/\Z^2$.

\end{Thm}
\begin{proof}
It follows from Proposition \ref{SYZ for av} and  Theorem \ref{GHside}.
Indeed, we obtain the integral affine structures on $\R^2/\Z^2$ as follows:
In non-Archimedean SYZ Picture, the integral affine structure of $\R^2/\Z^2\cong N_\R/L$ is given by
the inclusion $\tb: L\to N$.
In the Gromov-Hausdorff limit Picture, that of $\R^2/\Z^2$ is given by the matrix $B(l_i,l_j)$
up to scaling.
Hence, these two Pictures give the same integral affine structure to $N_\R/L$ up to scaling.
\end{proof}

\begin{Thm}[{\cite[Conjecture 3]{Kontsevich2006}} for Kummer Surfaces]\label{main}
  Under the setting \eqref{setting2}, we use the same notation as Corollary \ref{GHside for Kummer}.
Then the smooth locus of the IAMS induced by non-Archimedean SYZ Picture coincides with
  that of the IAMS induced by the Gromov-Hausdorff limit Picture up to scaling.
  That is,
  $\Sk (\cX)$ and $\nabla_B'(0)$ give the same IAMS structure (up to scaling) to
the $2$-sphere $S^2\cong (S^1\times S^1) /\{\pm 1\}$.
   In particular, the singular locus of the IAMS is $Z=\frac{1}{2}L/\Gamma=\{{\rm 4pts}\}$.
\end{Thm}
\begin{proof}
It follows from Theorem \ref{nAside}, Corollary \ref{GHside for Kummer} and Theorem \ref{Conj3 for as}.
Indeed, those two IAMS structures are the quotient of $N_\R/L$ by $H$.
Hence, these two Pictures give the same integral affine structure to $N_\R/\Gamma$ up to scaling.
\end{proof}

\begin{Rem}
  We note that, in the non-Archimedean SYZ Picture, we were implicitly rescaling
the affine structure by taking a
   base change $f:S'\to S$ as in \eqref{setting2}.
\end{Rem}

\bibliographystyle{amsalpha}
\providecommand{\bysame}{\leavevmode\hbox to3em{\hrulefill}\thinspace}
\providecommand{\MR}{\relax\ifhmode\unskip\space\fi MR }
\providecommand{\MRhref}[2]{%
  \href{http://www.ams.org/mathscinet-getitem?mr=#1}{#2}
}
\providecommand{\href}[2]{#2}

\end{document}